\documentclass{amsart}
\usepackage{amsmath,amsfonts,amssymb,amsthm,lscape,delarray,array,graphics}
\newtheorem{thm}{Theorem}
\newtheorem*{thma}{Theorem A}
\newtheorem*{thma*}{Theorem A*}
\newtheorem{lem}{Lemma}

\newtheorem{dfn}{Definition}
\newtheorem{cor}{Corollary}

\begin{document}
\title{A Solvable Version of the Baer--Suzuki Theorem}

\author{Simon Guest}
\address{Department of Mathematics\\
         University of Southern California\\
         Los Angeles, CA 90089--2532}
\email{sguest@usc.edu}
\begin{abstract}
Suppose that $G$ is a finite group and $x \in G$ has prime order
$p \ge 5$. Then  $x$ is contained in the solvable radical of $G$,
$O_{\infty}(G)$, if (and only if) $\langle x,x^g \rangle$ is
solvable for all $g \in G$. If $G$ is an almost simple group and $x \in
G$ has prime order $p \ge 5$ then this implies that there exists $g \in G$ such that $\langle x,x^g \rangle$ is not
solvable. In fact, this is also true when $p=3$ with very few
exceptions, which are described explicitly.
\end{abstract}
\maketitle
\section{Introduction}
The Baer--Suzuki theorem provides a useful characterization of the Fitting subgroup of a finite (or linear) group. It can be stated as follows:
\begin{thm}{\rm(Baer--Suzuki)}
Let $G$ be a finite (or linear) group. Suppose that for some $x
\in G$, $\langle x, x^g \rangle$ is nilpotent for all $g \in G$.
Then $\langle x^G \rangle$ is nilpotent. That is, $x$ is contained
in the Fitting subgroup of G.
\end{thm}
It is natural to ask if there is an analogous result if the
nilpotency condition is replaced with solvability. However, it is easy to
find counterexamples. For example, any two involutions generate a
dihedral group. So if $G$ is a non-abelian simple group and $x$ is an
involution in $G$ then $\langle x^{G} \rangle=G$ is not solvable
yet $\langle x, x^g \rangle$ is solvable for all $g \in G$. \par
There are also counterexamples when $x$ has order $3$.
Suppose that $x \in SL(n,3)$ ($n \geq 3$) has order $3$ and acts
trivially on some hyperplane; that is, $x$ is a transvection. Then
$x$ and any conjugate $x^g$ generate a group that acts trivially
on a subspace of codimension at most $2$.  Thus $\langle x, x^g
\rangle$ is solvable since it has a normal abelian subgroup $N$ such that $\langle x,x^g \rangle/N$ is isomorphic to a subgroup of
$GL(2,3)$. However, since $x$ is not central, it is not contained in the solvable radical of $SL(n,3)$ and $\langle x^G\rangle$
is not solvable.  The aim is to prove the following:
\begin{thma}
Let $G$ be a finite group. Suppose that $x \in G$ has prime order $p \ge 5$.
 If $\langle x, x^g \rangle$ is solvable for all $g \in G$ then
$\langle x^G \rangle $ is solvable. Equivalently, if $x \not\in
O_{\infty}(G)$ then there exists $g \in G$ such that $\langle x,
x^g \rangle$ is not solvable.
\end{thma}

It is worth noting that Theorem A implies the following result:
\begin{cor} \label{easycor}
Let $G$ be a finite (or linear) group. Then $G$ is solvable if and only if any two conjugates generate a solvable group. 
\end{cor}
\begin{proof} 
Let $G$ be a minimal counterexample to the version of the theorem for finite groups. Thus $G$ is a finite simple group by minimality and therefore $G$
contains an element $x$ of prime order $p \ge 5$.  So Theorem A implies that there exists $g \in G$ such that $\langle x,x^{g} \rangle$ is not solvable and thus $G$ is not a minimal counterexample. The version of the theorem for linear groups follows from the finite group version using a standard argument (see \cite[Corollary 1.2]{GGFar} for example).
\end{proof}
Also note that a minimal counterexample in Theorem~\ref{easycor} must be one of the minimal simple groups described by Thompson in the $N$-group paper \cite{Thompson}. Thus one could prove Theorem~\ref{easycor} without relying on the full Classification theorem by ruling out all of the minimal simple groups.  

Theorem A is also used in \cite{GGFar} to prove:
\begin{thm}\label{GGF1}
Let $G$ be a finite or linear group. Then $x \in G$ is contained in the solvable radical of $G$ if and only if $\langle x, x^{g_{1}}, x^{g_{2}}, x^{g_{3}} \rangle$ is solvable for all $g_{1}, g_{2}, g_{3} \in G$. 
\end{thm}
The proof in \cite{GGFar} relies on the Classification of Finite Simple Groups, however a weaker version of the theorem for finite groups is also given in \cite{GGFar} that does not rely on the Classification theorem:
\begin{thm}\label{GGF2}
Let $G$ be a finite group. Then $x \in G$ is contained in the solvable radical of $G$ if and only if every 7 conjugates of $x$ generate a solvable group.
\end{thm}
Theorems~\ref{GGF1} and \ref{GGF2} were announced in \cite{GPS} (see Theorems 7.3 and 7.4). Furthermore, Theorem A and Theorem~\ref{GGF1} have been obtained independently in \cite{GGKP} and \cite{GGKP2}, also using the Classification theorem.

\section{Reduction}

 Lemma \ref{reduction} below simplifies matters considerably. It reduces the proof to
 a situation where $G$ is an almost simple group.
\begin{lem} \label{reduction}
Suppose that $G$ is a finite group such that the Fitting subgroup
$F(G)$ is trivial. Let $L$ be a component of $G$. \\
{\rm (a)} If $x$ is an element of $G$ such that $x \not\in N_G(L)$ and
$x^2 \not\in C_G(L)$ then there exists an element $g$ in $G$ such
that $\left\langle x,x^g \right\rangle$ is not solvable. \\
{\rm (b)} If  $x$ is an element of $G$ such that $x \not\in N_G(L)$ and
$x^2 \in C_G(L)$ then there exist elements $g_1$ and $g_2$ in $G$
such that $\left\langle x,x^{g_1}, x^{g_2} \right\rangle$ is not
solvable.
\end{lem}
\begin{proof}
  Write $E(G)$ for the subgroup of $G$ generated by its
  components. Then the generalized Fitting subgroup is
  $F^{*}(G)=E(G)F(G)$. Since $F(G)=1$, it follows that
  $Z(F^*(G))=Z(E(G))$ is a normal abelian subgroup of $G$ and is
  therefore trivial. Also, $Z(E(G))$ is generated by the centers
  of each component of $G$ and so all of the components of $G$ are
  simple. Moreover, $E(G)$ must be a direct product of the
  components of $G$. So $G$ is embedded in  ${\rm Aut}(F^*(G)) =
  {\rm Aut}(E(G))$. It suffices to assume that $G =  \langle L,x
  \rangle$. Thus if $t:= |\{ L^{x^i} \: : \: \mbox{for } i=1,2,\dots
  \}|$ then $E(G)=L \times \cdots \times L^{x^{t-1}}$ and ${\rm
  Aut}(E(G)) \cong {\rm Aut}(L) \wr S_t$. Since $x$ does not
  normalize $L$, it follows that $t \ge 2$. Moreover, it suffices
  to assume that $x=(\sigma_1, \ldots ,\sigma_t) \tau$ where
  $\sigma_i \in {\rm Aut}(L^{x^{i-1}})$ and $\tau$ is the
  $t$-cycle $(12 \cdots t)$. Now observe that
\begin{align*}
  x^{(u_1, \dots, u_t)}& = (u_1, \dots, u_t)(\sigma_1, \dots ,
  \sigma_t)\tau(u_1, \dots,
  u_t)^{-1}\tau^{-1} \tau \\
  & = (u_1, \dots, u_t)(\sigma_1, \dots , \sigma_t)(u_t^{-1},
  u_1^{-1}, \ldots, u_{t-1}^{-1})\tau.
\end{align*}
So if
\begin{displaymath}
 \begin{split}
   & u_t =1 , u_{t-1}= \sigma_t, u_{t-2} = (\sigma_t \sigma_{t-1}),
   u_{t-3} = (\sigma_t \sigma_{t-1} \sigma_{t-2}), \ldots , \\
   & u_1 = (\sigma_t \sigma_{t-1} \cdots \sigma_{1}).
 \end{split}
\end{displaymath}
then $x^{(u_1, \dots, u_t)}= (y,1,1, \ldots,1)\tau $ for some $y
\in {\rm Aut}(L)$.  Thus, it suffices to assume that $x$ is of this
form. \\
Now let $g:=(w_1,\ldots, w_t) \in {\rm Aut}(L) \times \cdots
\times {\rm Aut}(L^{x^{t-1}})$ so that
\begin{displaymath}
  x^{-1}(w_1, \ldots, w_t)x(w_1^{-1}, \ldots, w_t^{-1})=
  (w_2w_1^{-1},\ldots )
\end{displaymath}
and
\begin{displaymath}
  (w_1, \ldots , w_t)x (w_1, \ldots , w_t)^{-1} x^{-1} = (w_1 y
  w_t^{-1} y^{-1}, \ldots)
\end{displaymath}
First, suppose that $t \ge 3$. By \cite[Theorem B]{GurAs}, there exist $l_1$ and $l_2$ in $L$ such that $L=\langle l_1, l_2 \rangle$. So define $w_1=1$, $w_2=l_1$, and 
$w_t=y^{-1}l_2y$. Thus $\langle x,x^g \rangle$ contains $(l_1,\ldots)$ and $(l_2,\ldots)$ and is not solvable. If $t=2$ and $x^2 \not\in C_G(L)$, then $x=(y,1)\tau$ and since 
$x^2=(y,y)$, it follows that $y \not=1$. Now $\left \langle y,L  \right\rangle$ is almost simple so by \cite{GK} there exists $z \in \left \langle y,L  \right\rangle$ such that 
$\left \langle y,z \right\rangle$ contains $L$. Observe that there exists $l \in L$ such that $z=y^{k}l$. So define $w_1:=1$ and $w_2:=l$ and then
\begin{displaymath}
  \begin{split}
    x^{2k-1}x^{(w_1,w_2)} = & (y^k,y^{k-1})\tau(w_1,w_2)(y,1)\tau(w_1^{-1},w_2^{-1})
    \\
    = & (y^kw_2w_1^{-1},\cdot)  =  (z,\cdot)
  \end{split}
\end{displaymath}
and so $\langle x,x^{(w_1,w_2)}\rangle$ cannot be solvable. This proves part
(a). \\
\indent To prove (b), suppose that $x$ does not normalize $L$ and $x^2 \in
C_G(L)$. So it suffices to assume that $t=2$ and $x= \tau$. If
$g_1:=(1,l_1)$ and $g_2:=(1,l_2)$ then
\begin{displaymath}
  x^{-1}x^{g_1}= (l_1,\cdot); \: x^{-1}x^{g_2}= (l_2,\cdot)
\end{displaymath}
and thus $\langle x,x^{g_1}, x^{g_2} \rangle$ is not solvable.
This proves part (b) of Lemma \ref{reduction}.
\end{proof}

\begin{lem} Suppose that $(x,G)$ is a
minimal counterexample. Then $G$ is almost simple.
\end{lem}
\begin{proof}
Since $(x,G)$ is a minimal counterexample, the solvable radical of $G$ is trivial. Let $N$ be a minimal normal subgroup. So $N \cong L \times \cdots \times L$ for some non-abelian simple group $L$. If $x \in N$ then $G=N$ since otherwise $ \langle x^N \rangle$ would be a solvable normal subgroup of $N$, and $N$ does not have any such subgroups. Thus, if $x \in N$ then $G$ is simple since $G$ has no non-trivial normal subgroups. Now assume that $x \not\in N$ and let $H:= \langle x,N  \rangle$. If $G \ne H$ then $ \langle x^H \rangle \cap N$ is a solvable normal subgroup of $N$ and is thus trivial. Thus $[x,N]=1$, which is not possible, because it would follow that $ [ \langle x^G   \rangle , N] = 1$. Since $N$ is a minimal normal subgroup,  $\langle x^{G}\rangle \cap N$ would be trivial and thus $\langle (xN)^{G/N} \rangle \cong  \langle x^G   \rangle N/ N \cong  \langle x^G   \rangle$. This is not possible since $\langle (xN)^{G/N}\rangle$ is solvable by minimality.  So $G=H= \langle x,N \rangle$. Note that the Fitting subgroup of $G$ is trivial since the solvable radical is trivial and thus $x$ normalizes every component by Lemma~\ref{reduction}. So $L$ is normal in $G$, $N=L$ and $G =  \langle x,L \rangle$.  Now $G$ is almost simple since $L$ is the unique minimal normal subgroup of $G$.
\end{proof}
The Classification of Finite Simple Groups can be used to determine
the possibilities for the socle $G_0$ of $G$, and thus eliminate
each possibility case by case. In fact, the following theorem is
slightly stronger and implies Theorem A.

\begin{thma*}
Let G be a finite almost simple group with socle $G_0$.
Suppose that $x$ is an element of odd prime order in $G$. Then one
of the following holds. \\
  {\rm (i)} There exists $g \in G$ such that $\langle x , x^g \rangle $ is
  not solvable. \\
  {\rm (ii)} $p=3$ and $(x,G_0)$ belongs to a short list of exceptions given
  in Table \ref{exceptions}. Moreover, there exist $g_1, g_2 \in G$
  such that $\langle x , x^{g_1}, x^{g_2} \rangle$ is not solvable,
  unless $G_0 \cong PSU(n,2)$, $PSp(2n,3)$. In any case, there exist
  $g_1,g_2,g_3 \in G$ such that
  $\langle x , x^{g_1}, x^{g_2}, x^{g_3} \rangle$ is not solvable.\\
\end{thma*}
\begin{table}
\begin{center}
  \begin{tabular}{|c|c|}
    \hline
    $G_0$ & $x$ \\
    \hline
    $PSL(n,3)$, $n>2$ & transvection  \\
    $PSp(2n,3)$, $n>1$ & transvection \\
    $PSU(n,3)$, $n>2$ & transvection \\
    $PSU(n,2)$, $n>3$ & reflection of order $3$\\
    $P \Omega^{\epsilon}(n,3)$, $n>6$ & $x$ a long root element \\
    $E_l(3),F_4(3), {^2}E_6(3), {^3}D_4(3)$ & $x$ a long root element \\
    $G_2(3)$ & $x$ a long or short root element \\
    $G_2(2)^{\prime} \cong PSU(3,3)$ & transvection \\
    \hline
  \end{tabular}
\vspace{5mm} \caption{List of exceptions to Theorem A*}
\label{exceptions}
\end{center}
\end{table}

\begin{cor}
Let $G$ be an almost simple group, and suppose that $x \in G$ has prime order $p \ge 5$. Suppose that $x$ is contained in the solvable radical of all proper subgroups $M$ containing $x$. Then there exists $g \in G$ such that $\left \langle x,x^g  \right \rangle = G$.
\end{cor}
\begin{proof}
By Theorem A*, there exists $g\in G$ such that $\left \langle x,x^g  \right \rangle$ is not solvable. If $\left \langle x,x^g  \right \rangle \ne G$ then it is contained in some maximal subgroup $M$. However, the hypothesis implies that $x \in  O_{\infty}(M)$ which would mean that $\left \langle x,x^g  \right \rangle$ would be solvable. Thus $\left \langle x,x^g  \right \rangle = G$.
\end{proof}
Clearly, we only need to check that the hypothesis in the corollary is true for all {\it maximal} subgroups. Indeed, if $x \in M$ and $M < M^{\prime} < G$ then $ \langle x^{M^{\prime}}  \rangle$ is solvable, therefore $\left \langle x^M  \right \rangle$ is solvable and thus $x \in O_{\infty}(M)$.
\section{Preliminaries}

Let $\overline{G}$ be a simple classical algebraic group of
adjoint type over the algebraic closure of $\mathbb{F}_q$. Let
$\sigma$ be a Frobenius morphism of $\overline{G}$ such that
$\overline{G}_{\sigma}:=\{g \in \overline{G} \: : \:
g^{\sigma}=g\}$ is a finite almost simple classical group over
$\mathbb{F}_q$. Write $G_0$ for the socle of
$\overline{G}_{\sigma}$ and note that $\overline{G}_{\sigma}$ is
the group ${\rm Inndiag}(G_0)$ of inner diagonal automorphisms of
$G_0$. A collection of lemmas, definitions, and theorems are listed below, which will be very useful in the sequel:

\begin{lem} \label{lift}
 Let $x \in \overline{G}_{\sigma}$ have odd prime order $r$.
 Define $(G,\hat{G})$ as follows:
  \vspace{+1mm}
\begin{center}
 \begin{tabular}{|c|ccc|}
   \hline
   $G_0$  & $PSL^{\epsilon}_n(q)$ & $PSp_n(q)$ & $P\Omega_n^{\epsilon}(q)$ \\
   \hline
$(G, \hat{G})$ & $(\overline{G}_{\sigma}, GL^{\epsilon}_n(q))$
&$(G_0, Sp_n(q))$ & $(G_0, \Omega_n^{\epsilon}(q))$ \\
   \hline
\end{tabular}
\end{center}
\vspace{+1mm}
 {\rm (a)} Then one of the following holds: \\
{\rm (i)} $x$ lifts to an element $\hat{x} \in \hat{G}$ of order r such
that $|x^G| = |\hat{x}^{\hat{G}}|$; \\
{\rm (ii)} $G_0 = PSL^{\epsilon}_n(q)$ , $r \mid \gcd(q- \epsilon,n)$ and $x$
is $\overline{G}$-conjugate to $[I_{\frac{n}{r}},\omega
I_{\frac{n}{r}}, \ldots, \omega^{r-1} I_{\frac{n}{r}}]$ where
$\omega$ is a primitive $r$th root of unity. \\
{\rm (b)} If $r \nmid q$ then $x^{G_0} = x^{\overline{G}_{\sigma}}$.
\end{lem}

\begin{proof}
  See \cite[3.11]{Bs} and \cite[4.2.2(j)]{GLS}
\end{proof}
\begin{dfn}
Let ${\mathcal A}$ be the set of pairs $(x,H)$ such that:
\begin{itemize}
 \item[{\rm (i)}] $x$ is an element of odd prime order contained in a
group $H$;
\item[{\rm (ii)}] $H/O_{\infty}(H)$ is almost simple;
\item[{\rm (iii)}] $x$ is not contained in $O_{\infty}(H)$; \item[(iv)]
$(x,H/O_{\infty}(H))$ is not one of the examples in Table
\ref{exceptions}.
\end{itemize}
\end{dfn}
\begin{lem} \label{InnDiag} \label{Autl} \label{as}
 {\rm(a)} If $x \in G$ is an inner-diagonal automorphism of $G_0$ and $|x^{G_0}|=
  |x^{\overline{G}_{\sigma}}|$ then it suffices to take $G =\overline{G}_{\sigma}$. \\
  {\rm(b)} If y is some $Aut(G_0)$-conjugate of $x$ and there exists
  $l \in G_0$ such that $\langle y,y^l \rangle$ is not solvable then
  there exists $l^{\prime}\in G_0$ such that $\langle x,x^{l^{\prime}} \rangle$
  is not solvable. \\
  {\rm(c)} If $x$ is contained in $H$, a proper subgroup of $G$, and
  $(x,H) \in {\mathcal A}$ then $G$ cannot be a minimal counterexample to
  Theorem A*.
\end{lem}

\begin{proof}
  (a) Suppose that the theorem is true for $\overline{G}_{\sigma}$. If
  $x$ is contained in $G$ then $x \in \overline{G}_{\sigma}$ and so there exists $g \in
  \overline{G}_{\sigma}$ such that $\langle x^g ,x \rangle$ is not
  solvable. But then there exists $g_1 \in G_0$ such that $x^{g_1} = x^g$ by the condition. \\
  (b) Suppose that $y=x^g$ for some $g \in Aut(G_0)$. Then $\langle
  y,y^l \rangle^{g^{-1}}= \langle x,y^{lg^{-1}} \rangle=\langle
  x,x^{l^{\prime}}
  \rangle$ since $lg^{-1}=g^{-1}glg^{-1}=g^{-1}l^{\prime}$.\\
  (c) Trivial.
\end{proof}

\begin{lem} \label{count}
Let $X_1, \ldots X_k$ be representatives for the conjugacy classes
of maximal subgroups containing $x$. Let $n_i$ be the number of
conjugates of $X_i$ that contain $x$. If
\begin{displaymath}
  |x^G|^2  > \sum_i n_i |x^G \cap X_i| = \sum |x^G
  \cap X_i|^2[G:X_i].
\end{displaymath}
then there exists $g \in G$ such that $\langle x , x^g \rangle =G$
\end{lem}

\begin{proof}
  Let $X_{i1}, \ldots , X_{in_i}$ be the conjugates of $X_i$ that
  contain $x$.  The aim is to show that $x^G$ cannot be contained
  in $\cup_{i,j} X_{ij}$, since this proves the lemma.  It is not
  hard to show that $n_i/[G:X_i]= |x^G \cap X_i| / |x^G|$. It then
  follows that
\begin{displaymath}
  \begin{split}
    |x^G \cap \cup_{i,j} X_{ij}|  &\le \sum_i n_i |x^G \cap X_i| \\
    &= \sum |x^G \cap X_i|^2[G:X_i]/|x^G|
  \end{split}
\end{displaymath}
and so if $x^G$ were contained in $\cup_{i,j} X_{ij}$ then
\begin{displaymath}
  \begin{split}
    |x^G| = |x^G \cap \cup_{i,j} X_{ij}|  & \le \sum_i n_i |x^G \cap X_i| \\
    &= \sum |x^G \cap X_i|^2[G:X_i]/|x^G|.
  \end{split}
\end{displaymath}
However, this implies that
\begin{displaymath}
  |x^G|^2  \le \sum_i n_i |x^G \cap X_i| = \sum |x^G
  \cap X_i|^2[G:X_i],
\end{displaymath}
which contradicts the hypothesis.
\end{proof}
{\flushleft \textit{Remark}} If
\begin{displaymath}
  |G|/|C_{G}(x)|^2 > \sum_{i} |x^G \cap X_i|
\end{displaymath}
then the conclusion of the theorem holds.\\

\begin{lem}\label{2.2}
  Let $G_0$ be a simple group of Lie type and suppose that G satisfies
  $G_0 \trianglelefteq G \le \rm{Inndiag}(G_0)$. \\
  {\rm (i)} Suppose that $x \in G$ is unipotent and $P_1$ and $P_2$ are distinct maximal parabolic subgroups containing a common Borel subgroup, with unipotent radicals $U_1$ and $U_2$ respectively. Then there exists $i \in \{1,2\}$ such that $x$ is $G$-conjugate to an element of $P_i \backslash U_i$. \\
  {\rm (ii)} Suppose that $x \in G$ is semisimple and is contained in a parabolic subgroup of $G$. Suppose further that the Lie rank of $G_0$ is at least $2$. Then there exists a maximal parabolic subgroup $P$ with a Levi complement $J$ such that $x$ is conjugate to an element of J not centralized by any component of J.
\end{lem}
\begin{proof}
  See \cite[Lemma 2.2]{GS}.
\end{proof}

\begin{thm} \label{2M} Let G be an almost simple group and let $x \in
  G$ with $x \not= 1$.  If $x^G \subseteq M_1 \cup M_2$ for subgroups
  $M_1$ and $M_2$ of $G$ then $G_0$ is contained in $M_i$ for $i=1$ or
  $2$.
\end{thm}

\begin{proof}
  See \cite[Theorem 2.1]{Gu}.
\end{proof}
To begin the proof of Theorem A*, let $(x,G)$ be a
minimal counterexample. Then $G$ is almost simple with socle
$G_0$. If $p \ge 5$ then Theorem A holds for any group
containing fewer elements than $G$.

\section{Alternating Groups}

Suppose that $G_0=A_n$. Then $x$ is contained in $A_n$ since it
has odd order. Firstly, consider the case where $p \ge 5$. The
cycle structure of $x$ will consist only of $p$-cycles. So it
suffices to assume that $x=(12\ldots p)\sigma$ for some $\sigma
\in {\rm Alt}\{p+1,\ldots, n\}$. Observe that if $g:=(123)$ then
$xgx^{-1}g^{-1}=(2p3)$. Thus $\langle x, x^g \rangle$
contains ${\rm Alt}\{1,2,\ldots, p\}$ since a primitive permutation group of degree $p \ge 5$ containing a $3$-cycle contains $A_{p}$ (see \cite[Theorem 13.9]{Wiel} for example). So $(x,G)$ cannot be a
counterexample in this case.\\ 
\indent Now suppose that $p=3$. Then the cycle
structure of $x$ consists of only $3$-cycles. If $x$ is the product of more than one $3$-cycle then it suffices to assume that $G=A_{6}$ and $x$ is the product of two $3$-cycles. But then
$x$ is conjugate to a $3$-cycle in ${\rm Aut}({A_{6}})$. Thus we may assume that $x$ is a $3$-cycle in $A_{5}$ and without loss of generality, that $x=(123)$. If $g:=(1 4 2
5 3)$ then $x^g = (451)$.  Thus, $x x^g = (12345)$ and $\langle x,
x^g \rangle \cong A_5$.



\section{$PSL(n,q)$}

If $G_0 \cong PSL(n,q)$ then it is convenient to treat the cases
where $n=2$ and $n \ge 3$ separately.

\subsection{$G_0 \cong PSL(2,q)$}

Suppose that $x$ is in ${\rm Inndiag}(PSL(2,q)) \cong PGL(2,q)$.
Since $x$ has odd order, it must lie in $PSL(2,q)$.

\subsubsection{$x \in {\rm Inndiag}(PSL(2,q))$ and $p \mid q$}

If $p \mid q$ and $p \ge 5$ then it suffices to assume that $x$ is contained in $PSL(2,p)$. Consider the possibilities for the maximal subgroups of $PSL(2,p)$ containing $\langle x, x^g \rangle$, which are described in \cite[Theorem 6.5.1]{GLS}. By the order of $x$ and since $(x,G)$ is a minimal counterexample, the only type of maximal subgroup possible is a Borel subgroup, $B$. Now since $p \mid q$, $x$ and $x^g$ must lie inside the kernel $K$ of $B$ which is (elementary) abelian. So any $p$-elements lying in a common Borel subgroup must commute. Thus, since there must exist a conjugate of $x$ that does not commute with $x$---otherwise $[x^G,x^G]=1$ and $G_0$ would be abelian---there exists $g \in G$ such that $\langle x , x^g \rangle = PSL(2,p)$, which is not solvable for $p \ge 5$.  \\
\indent If $p=3$ then $q=3^a$, where $a >1$ and since $x \in PSL(2,q)$, it suffices to assume that $G=PSL(2,q)$. Now $A_6 \cong PSL(2,9)$ so let us assume that $q>9$. If $q=3^a$ and $a$ is not prime then there exists a conjugate $x^g$ of $x$ that is contained in a subfield subgroup $H$ with $(x^g,H)$ in ${\mathcal A}$. So $a$ must be an odd prime. Now we may assume that 
\[ x= \left( \begin{matrix} 1 & 1 \\
0 &1
\end{matrix} \right). \]
There are two classes of transvections in $G$, and since $-1$ is not a square in $\mathbb{F}_{q}$, $x$ and $x^{-1}$ are not conjugate. Thus if we let
\[ y:= \left( \begin{matrix} 1 & 0 \\
s &1
\end{matrix} \right)\] 
then $x$ or $x^{-1}$ is conjugate to $y$. So there exists $g \in G$ such that $\left \langle x,x^{g} \right \rangle$ contains 
\[ xy= \left( \begin{matrix} 1+s & 1 \\
s&1
\end{matrix} \right), \]
which is semisimple and has trace $s+2$. In particular we can choose $s$ so that $xy$ has order $\tfrac{q+1}{2}$ and an inspection of the maximal subgroups of $G$ shows that $\left \langle x,x^{g} \right \rangle=G$.

\subsubsection{$x \in {\rm Inndiag}(PSL(2,q))$ and $p \nmid q$}

Suppose now that $p\nmid q$. Then either $p \mid q-1$ or $p \mid q+1$. If
$p \mid q-1$ then $x$ is contained in a split torus. Examining the
character table of $PSL(2,q)$ shows that for an element $z$ of
order $(q+1)/(2,q-1)$,
\begin{displaymath}
    \sum_{\chi \in {\rm Irr}(G_0)} \frac{\overline{\chi(z)}|\chi(x)|^2}{\chi(1)}
\not=0
\end{displaymath}
so there exists $g \in G$ such that $[x,g]$ has order $(q+1)/(2,q-1)$. That is $[x,g]$ generates a non-split torus. It follows that $\langle x,x^g\rangle$ generates $PSL(2,q)$. Indeed, $\langle x,x^g\rangle$ contains an irreducible torus, and it also contains $x$, which does not normalize this torus. An inspection of the maximal subgroups of $PSL(2,q)$ yields that $\langle x,x^g\rangle$ must generate the whole group for $q \ge 11$, and $q=8$. It suffices to assume that $q \not= 4$, $5$ or $9$ since in those cases $G$ is isomorphic to an alternating group. When $q=7$, the normalizer of a non-split torus is not maximal, but is contained in subgroups isomorphic to $S_4$. However, since $p \mid q-1$, $\langle x,x^g\rangle$ cannot be contained in $S_4$ since $S_4$ does not contain two elements of order $3$ whose product has order $4$. \\
\indent If $p \mid q+1$ then the character table implies that there
exists $g \in G$ such that $[x,g]$ has order $(q-1)/(2,q-1)$. Thus
$\langle x,x^g\rangle$ contains a split torus, and since $p \mid q+1$,
it acts irreducibly. Therefore, an inspection of the maximal
subgroups shows that for $q \ge 13$ and $q=8$, $\langle
x,x^g\rangle=PSL(2,q)$. Again, the cases when $q=4$, $5$, and $9$
do not concern us. Also note that $q\not=7$ since $p \mid q+1$. If
$q=11$ then $\langle x,x^g\rangle$ contains a maximal split torus
and acts irreducibly. The list of maximal subgroups then implies
that either $\langle x,x^g\rangle=PSL(2,q)$ or $A_5$. There are no
other possibilities for $x \in {\rm Inndiag}({\rm PSL}(2,q))$.

\subsubsection{$x$ an outer automorphism of $PSL(2,q)$}

Suppose that $x$ is not contained in $\rm{Inndiag}(G_0)$. Then by
\cite[7.2]{GL}, and since $x$ has odd order, there exists an
element $g \in PGL(2,q)$ such that $x^g$ is a standard field
automorphism. So it suffices to assume that $x$ is a standard
field automorphism by Lemma \ref{Autl}, and moreover, that
$G= \langle G_0,x \rangle$. Write
$q=q_1^{p}$ and consider the set $\Gamma =\{ y \in x^{G_0} |
\langle x,y \rangle \not= G \}$. The aim is to bound the
cardinality of $\Gamma$ and show that this is smaller than
$|x^{G_0}|$. Now if $y \in \Gamma$ then consider the possibilities
for subgroups $H$ of $G_0$ containing $\langle x,y\rangle \cap
G_0$. Observe that $\langle x,y\rangle \cap G_0$ cannot be
dihedral. Indeed, since a dihedral group has a characteristic
cyclic subgroup of index $2$, $K$ say, $K$ would be normal in
$\langle x,y\rangle$. Now $\langle x,y\rangle/K$ has a normal
subgroup of  order $2$ and a subgroup of order $p$, which is
normal since it has index $2$. So, $\langle x,y\rangle/K$ is
abelian of order $2p$, but this is impossible since it is
generated by two elements of order $p$. Thus, it suffices to
assume that $H$ is a Borel subgroup, a cyclic group of order
$(q+1)/(2,q-1)$, or a subfield subgroup. Since $p$ is odd any
$A_5$ or $S_4$ will be contained in a subfield subgroup and any
cyclic group of order $(q-1)/(2,q-1)$ will be contained in a Borel
subgroup. Now let $H$ be a Borel, non-split torus or subfield
subgroup of the form $L(2,q^{1/r})$, where $r$ is a prime distinct
from $p$. Observe that we may assume that there are no subfield
subgroups of the form $L(2,q^{1/r})$, ($r \not= p$) since some
conjugate of $x$ will be a non-trivial field automorphism of the
simple subgroup $L(2,q^{1/r})$ contradicting the minimality of
$(x,G)$. Now the conjugates of $H$ fixed by $x$ form one
$C_{G_0}(x)$ orbit.  This follows from the fact that any two
conjugates of $x$ in $H \langle x \rangle$ are in fact conjugate
by an element of $H$, which is a consequence of Lang's Theorem
(see \cite[7.2]{GL}). So if $k$ is the number of conjugates of $H$
fixed by $x$ then
\begin{displaymath}
  \begin{split}
    |\{ y \in x^{G_0} : \langle x,y \rangle \cap G_0 \textrm{ is
      contained in a conjugate of $H$} \}| & \le |x^H|.k \\ & =
    \frac{|H||C_{G_0}(x)|}{|C_H(x)|^2}.
  \end{split}
\end{displaymath}
Moreover $x$ does not fix any non-trivial conjugate of $C_{G_0}(x)
= PSL(2,q^{1/p})$, so
\begin{displaymath}
  |\{ y \in x^{G_0}  :  \langle x,y \rangle \cap G_0
  \textrm{ is contained in some conjugate of } C_{G_0}(x) \}| \le
  |C_{G_0}(x)|.
\end{displaymath}
Therefore if the representatives for the conjugacy classes of the
subgroups above are denoted by $H_1$, $\ldots$ , $H_m$,
$H_{m+1}:=C_{G_0}(x)$, then
\begin{displaymath}
\begin{split}
  |\Gamma|  &= |\{ y \in x^{G_0} : G_0 \cap \langle x,y \rangle
  \textrm{ is contained in some conjugate of some $H_i$}
  \} |\\
  & \le |C_{G_0}(x)| + \sum_{i=1}^m |H_i||C_{G_0}(x)|/|C_{H_i}(x)|^2  .
\end{split}
\end{displaymath}
If $q_0:=q^{1/p}$ then
\begin{displaymath}
\begin{split}
  |H||C_{G_0}(x)|/|C_H(x)|^2 &= (q_0^p+1)q_0(q_0^{2}-1)/(q_0+1)^2 \\
  &= (q_0^p+1)q_0(q_0-1)/(q_0+1)
\end{split}
\end{displaymath}
when $H$ is a non-split torus. Similarly if $H$ is a Borel
subgroup then
\begin{displaymath}
  |H||C_{G_0}(x)|/|C_H(x)|^2 = q_0^p(q_0^p-1)(q_0+1)/q_0(q_0-1)
\end{displaymath}
So,
\begin{displaymath}
  |\Gamma| \le q_0(q_0^{2}-1) +
  \frac{q_0^p(q_0^p-1)(q_0+1)}{q_0(q_0-1)} +
  \frac{(q_0^p+1)q_0(q_0-1)}{(q_0+1)}
\end{displaymath}
However, $|x^{G_0}| = |G_0|/|C_{G_0}(x)| =
\frac{q_0^p(q_0^{2p}-1)}{q_0(q_0^{2}-1)}$ and $q \ge 8$ so it
follows that $|x^{G_0}| > |\Gamma|$ as required. Thus, if $x$ is
an outer automorphism of $PSL(2,q)$ then $(x,G)$ cannot be a
minimal counterexample.

\section{Outer Automorphisms}

If $(x,G)$ is a minimal counterexample and $x$ is an outer
automorphism of $G_0$ then the work for $G_0=PSL(2,q)$ allows a
considerable narrowing of the possibilities for $G_0$. This is
demonstrated in Lemma \ref{Outer} below.
\begin{lem} \label{Outer} If $x$ is an outer automorphism of $G_0$ that is not inner-diagonal 
  and $(x,G)$ is a minimal counterexample then $G_0$ is a Suzuki--Ree
  group.
\end{lem}
\begin{proof}
Since $x\not\in \rm{Inndiag}(G_0)$ and $x$ has odd prime
order, either $x$ is a field automorphism or, $G_0 \cong D_4(q)$
or ${^3}D_4(q)$ and $x$ is a graph or graph-field automorphism.
Since the case where $G_0 \cong PSL(2,q)$ has already been
eliminated the Lie rank is at least $2$. If $x$ is a field
automorphism then by \cite[7.2]{GL} and Lemma \ref{Autl} it
suffices to assume that $x$ is a standard field automorphism. So
if $G_0$ is not a Suzuki--Ree group then $x$ will act non-trivially
as a field automorphism on a fundamental $SL_2$-subgroup, by
\cite[Theorem 3.2.8]{GLS}. So $(x,G)$ cannot be a minimal
counterexample. \\
\indent If $G_0 \cong {^3}D_4(q)$ and $x$ is a graph automorphism
of order $3$ then \cite[9.1]{GL} describes the conjugacy classes
of such elements.  Let $\gamma$ be the standard triality
automorphism and $g = \overline{h_{\beta_0}(\omega)}$ where
$\omega$ is a primitive cube root of unity and $\beta_0$ is the
$\gamma$ invariant fundamental root. Thus, if $3 \nmid q$ then it
suffices to assume that $x$ is either $\gamma$ or $g \gamma$.
Also, if $3 \mid q$ then it suffices to assume that $x$ is either
$\gamma$ or $x_{\beta}(1) \gamma$ where $\beta$ is the highest
root. In all cases, $x$ normalizes the maximal parabolic
corresponding to $\beta_0$. Moreover $x$ acts non-trivially on the
Levi complement in all these cases and so $(x,G)$ cannot be a
minimal counterexample. The only case left is where $G_0 \cong
D_4(q)$ and $x$ is a graph or field-graph automorphism of order
$3$. In which case, using \cite{GL}, it suffices to assume that
$x$ is either the standard triality (and $C_{G_0}(x)=G_2(q)$) or
it normalizes but does not centralize a subgroup isomorphic to
$G_2(q)$. In the latter case $x$ induces a non-trivial
automorphism on $G_2(q)$, so $(x,G)$ cannot be a minimal
counterexample. In the former case, since $G_2(q)$ does not
contain a Sylow $3$-subgroup, $x$ normalizes more than one
conjugate of $G_2(q)$. Since it only centralizes one $G_2(q)$
subgroup, it follows that $x$ induces a non-trivial automorphism
on some subgroup isomorphic to $G_2(q)$ and so $(x,G)$ cannot be a
minimal counterexample.
\end{proof}

\section{$PSL(n,q)$, $n\ge3$}

\subsection{$x \in PGL(n,q)$, $p \nmid q$, $n \ge 3$}

Now suppose that $(x,G)$ is a minimal counterexample with
$G_0=PSL(n,q)$ and $n \ge 3$.

\begin{lem}
  For $n \geq 3$, if one can lift $x$ to an element of order $p$
  in $GL(n,q)$ and $x$ does not act irreducibly then $(x,G)$
  cannot be a minimal counterexample.
\end{lem}

\begin{proof}
  Suppose that one can lift $x$ to an element of $GL(n,q)$ order
  $p$. Now the minimal polynomial $m_x(t)$ of $x$ divides $(t^p -
  1)$ so suppose that $(t^p - 1)/(t-1)$ factors into irreducibles
  $g_1(t) \ldots g_k(t)$.  Then each non-linear $g_i(t)$ is the
  minimal polynomial of some primitive $p$th root of unity
  $\zeta_p$. Thus
  \begin{displaymath}
    \rm{deg} \; g_i (t) =
    [\mathbb{F}_q(\zeta_p) : \mathbb{F}_q].
  \end{displaymath}
  But $\mathbb{F}_q(\zeta_p)$ is just the finite field of $q^e$
elements where $e$ is the smallest positive integer such that $p \mid 
q^e-1$. So all of the $g_i(t)$'s have degree $e$.  Now $m_x(t)$ is
a product of some $g_i(t)$'s and possibly $t-1$. By considering
the rational canonical form of $x$, it is clear that there is an
$e$-dimensional subspace $U$ of $V$ on which $x$ acts invariantly,
non-trivially and irreducibly. If $2 \le e < n$ then consider the
induced transformation of $U$, $x_U$ so that $x_U \in GL(e,q)$.
Now observe that if $(e,q) \not=(2,2)$ or $(2,3)$ then
$(x_U,GL(e,q)) \in {\mathcal A}$. If $(e,q)=(2,3)$ then $p$ would
be $2$. So the only case of concern is $(e,q)=(2,2)$ and then for
$n \ge 4$ one can just reduce to the case where $G_0= PSL(4,2)$.
However $PSL(4,2) \cong A_8$ and $PSL(3,2)\cong PSL(2,7)$, which
have already been eliminated.  If $e=1$ then since $p \mid q-1$, $q \ge
4$. Now $x$ will act non trivially on a 2 dimensional subspace
$U^{\prime}$; thus $x_{U^{\prime}} \in GL(2,q)$ and
$(x_{U^{\prime}},GL(2,q)) \in {\mathcal A}$. So unless $e=n \ge
3$, $(x,G)$ cannot be a minimal counterexample.
\end{proof}
Now observe that the proof above shows that if $(x,PGL(n,q))$ is a
minimal counterexample and $x$ lifts to an element of order $p$ in
$GL(n,q)$ then $p$ is a primitive prime divisor of $q^n-1$ and $x$
acts irreducibly. Also, if $x$ acts irreducibly and $n$ is not
prime then some conjugate of $x$ is contained in a field extension
subgroup $PGL(\frac{n}{r},q^r)$. Thus, if $(x,PGL(n,q))$ is a
minimal counterexample then $n$ is prime.\\
\indent The results in \cite{GPPS} state that any subgroup of
$GL(n,q)$ which has order divisible by a primitive prime divisor
of $q^e-1$ must be one of nine types ($2.1$--$2.9$). The results
of \cite{GPPS} will be used frequently, and are summarized in
Table \ref{table:GPPS}. The notation of \cite{GPPS} will be used.
Namely, that the element of $GL(d,q)$ that is a primitive prime
divisor of $q^e-1$ be referred to as a ppd($d$,$q$,$e$)-element.
The only elements that are of interest are
ppd($n$,$q$,$n$)-elements where $n$ is (an odd) prime. So what are
the possibilities for a maximal subgroup $M$ of $GL(n,q)$
containing $x$?
\begin{table}
\begin{center}
  \begin{tabular}{|l|l|l|}
    \hline
    Type & Rough description&  Conditions on $d,q,e$ \\
    \hline
    Classical  ($2.1$(a)) & $SL(d,q_0) \unlhd M$ & $p$ a ppd($q_0$,$d$,$e$)-element \\
    Classical  ($2.1$(b))& $Sp(d,q_0) \unlhd M$ & $d$, $e$ both even;\\
    && $p$ a ppd($q_0$,$d$,$e$)-element\\
    Classical  ($2.1$(c))& $SU(d,q_0) \unlhd M$ & $q_0$ a square; $e$ odd; \\
    && $p$ a ppd($q_0$,$d$,$e$)-element\\
    Classical ($2.1$(d))& $\Omega^{\epsilon}(d,q_0) \unlhd M$ & $\epsilon=\pm$
    when $d$ even; \\
    &&  $\epsilon= 0$ when $dq$ is odd;\\
    && $e$ even;\\
    && $p$ a
    ppd($q_0$,$d$,$e$)-element\\
    Reducible  ($2.2$)& $M$ reducible & × \\
    Imprimitive ($2.3$)& $M \le GL(1,q)\; \rm{wr} \;S_d$ & $p=e+1 \le d$
    \\
    Extension Field  ($2.4$(a)) & $M \le GL(1,q^d).d$ & $p=d=e+1$\\
    Extension Field  ($2.4$(b)) & $M \le GL(d/b,q^b).b$ & $b \mid  {\rm gcd}(d,e)$\\
    Symplectic type  ($2.5$) & & $d=2^a$; \\
    && $q$ odd not a square;\\
    && $p=d+1=e+1$ or \\
    && $p=d-1=e+1$\\
    Nearly simple ($2.6$--$2.9$) & $M/(M \cap Z)$ simple & Possibilities listed \\
    && in tables in \cite{GPPS} \\
    \hline
  \end{tabular}
\vspace{+2mm} \caption{\label{table:GPPS} Summary of descriptions
in \cite{GPPS} of subgroup types containing
ppd($d$,$q$,$e$)-elements }
\end{center}
\end{table}

\begin{lem} \label{2.6}
Suppose that $x$ is a ppd($n$,$q$,$n$)-element contained in a
subgroup $M$ of $G$, where $G$ is a classical group of dimension
$n \ge 3$ and $(x,G)$ is a minimal counterexample. Then $p \ge 5$
and $M$ cannot be of type $2.2$, $2.3$, $2.4$(a), $2.6$, $2.7$,
$2.8$, or $2.9$.
\end{lem}
\begin{proof}
 Firstly, if $p=3$ then since $p\nmid q$, Fermat's Little Theorem
implies that $p \mid q^2-1$, thus $p$ cannot be a primitive prime
divisor of $q^n-1$ for $n \ge 3$. If $G$ is a classical group then
$M \le G \le GL(n,q)$ for some $q$, and so $M$ must be one the
examples in \cite{GPPS}. All of the subgroups $M$ of type
$2.6$--$2.9$ are almost simple modulo scalars so it suffices to
check that $(x,M/(M \cap Z)) \in {\mathcal A}$. If $M$ is of type
$2.6$ or $2.7$ then $F^*(M/(M \cap Z)) \cong A_d$ for some $d$, or
a sporadic group and so $(x,M/(M \cap Z)) \in {\mathcal A}$. The
only ppd($n$,$q$,$n$)-elements in type $2.8$ examples ($M/(M \cap
Z)) \in \mathit{Lie}(q_0)$ are with $M^{( \infty )}=G_2(q_1)$,
$q_0=2$ and $M^{(\infty)}={^2}B_2(q_1)$, $q_0=2$ but these
occurrences must all lie in ${\mathcal A}$. Similarly, all of the
type $2.9$ subgroups in \cite[Tables 7 and 8]{GPPS} coincide with
elements of ${\mathcal A}$. Since $x$ acts irreducibly it cannot
be contained in a reducible subgroup of type $2.2$ and it cannot
be contained in a type $2.3$ example since these are only examples
for ppd($d$,$q$,$e$)-elements where $e+1\le d$. Similarly $x$
cannot be contained in a $2.4$(a) type subgroup since these are
only examples for ppd($d$,$q$,$e$)-elements where $e+1 = d$.
\end{proof}
Suppose that $x$ is contained in a classical example of type
$2.1$. By \cite{KL} and \cite{GPPS}, since $n \ge 3$, all of the
classical examples containing ppd($n$,$q$,$n$)-elements are almost
simple modulo scalars. So if $x$ is contained in a type $2.1$
subgroup $M$ then $(x,G)$ cannot be a minimal counterexample since
$p \ge 5$. The symplectic type examples ($2.5$) only occur as
subgroups of $GL(2^a,q)$ but it is assumed that $n$ is an odd
prime. Therefore, the only possibilities for subgroups $M$
containing $x$ are the extension field examples of type
($2.4$(b)). Since $n$ is prime, $M$ must be of type $GL(1,q^n).n$.
Moreover, if $p \mid n$ then $p=n$ since $n$ is prime. However, $p\nmid
q^p-1$ so $p \nmid n$. Thus, $x$ must lie inside the Singer cycle
$GL(1,q^n)$. Furthermore, $C_{GL(n,q)}(x)=GL(1,q^n)$, thus $x$ can
only lie in one such maximal subgroup and applying Theorem
\ref{2M} yields that $(x,G)$ cannot be a minimal counterexample.
\begin{lem}
  If $x$ does not lift to an element of order $p$ in $GL(n,q)$
  then $(x,G)$ cannot be a minimal counterexample.
\end{lem}
\begin{proof}
  Suppose that $x$ does not lift to an element of order $p$ in
  $GL(n,q)$. Now $x^p$ is central so $x$ satisfies the polynomial
  $p(t):=t^p - \lambda$. Now $p(t)$ is irreducible over
  $\mathbb{F}_q$. For $p \mid (q-1)$, since $x$ does not lift, thus any
  field containing a root $\alpha$ of $p(t)$ would be a splitting
  field for $p(t)$. So the degree of any irreducible factor of
  $p(t)$ is the degree of the splitting field extension over
  $\mathbb{F}_q$. However, $p(t)$ has prime degree and so it is
  either irreducible or it splits completely. It cannot split
  completely otherwise $\lambda$ would have $p$th roots and $x$
  would lift to an element of order $p$. Thus, the irreducible
  module for $\langle x \rangle$ has dimension $p$ and so it
  suffices to deal with case where $n=p$. So let $v$ be a vector
  in $V$ and consider the action of $x$ on $v$. The vectors $v,
  xv, x^2v, \ldots , (x^{p-1})v$ form a basis for $V$ since $x$
  acts irreducibly. Moreover $x^pv=\lambda v$. So $x$ is contained in a
  subgroup of type $GL(1,q) \wr S_p$ and $x$ acts as a p-cycle in
  the $S_p$. So for $p \ge 5$, we have shown that $(x,G)$ cannot
  be a minimal counterexample.  Now suppose that $p=3$. Then it
  suffices to assume that $x$ has the form
\begin{displaymath}
  \left(
    \begin{matrix}
      0 & 0 &\lambda  \\
      1& 0 & 0 \\
      0& 1 & 0
    \end{matrix} \right)
\end{displaymath}
Now let $t^2 - \mu_2t - \mu_1$ be an irreducible polynomial in
$\mathbb{F}_q[t]$, such that $\left( \begin{matrix}
    0 & \mu_1  \\
    1& \mu_2
  \end{matrix} \right)$ has order $q^2-1$. Now $x$ is conjugate to
\begin{displaymath}
  y := \left(
    \begin{matrix}
      0 & 0& -\mu_1^{-1}\lambda  \\
      0& \mu_1 & \mu_2^{-1}(\mu_1^{-1}\lambda-\mu_1^2) \\
      1& \mu_2 & -\mu_1
    \end{matrix} \right)
\end{displaymath}
and therefore
\begin{displaymath}
  x^{-1}y := \left(
    \begin{matrix}
      0 & \mu_1& \mu_2^{-1}(\mu_1^{-1}\lambda-\mu_1^2)  \\
      1& \mu_2 & -\mu_1 \\
      0& 0 & -\mu_1^{-1}
    \end{matrix} \right)
\end{displaymath}
has order a multiple of $q^2-1$. Thus, by \cite[Theorem
6.5.3]{GLS}, $\langle x,y \rangle$ is not solvable and hence
$(x,G)$ cannot be a minimal counterexample. The case where $p \mid q$ is considered in the next section.

\section{Unipotent elements}
\begin{lem} \label{genunipotents}
Suppose that $G_0$ is a simple group of Lie type and suppose that
$x \in G_0$ is unipotent of order $p$. If $G_0$ is defined over
$\mathbb{F}_q$ and $q\not=3$ then $(x,G)$ cannot be minimal
counterexample unless $G_0 =  PSU(3,q)$ or ${^2}G_2(q)$.
\end{lem}
\begin{proof}
The case where $G_0 = PSL(2,q)$ has already been done. Since $p$
is an odd prime, $G_0 \not={^2}B_2(q)$ or ${^2}F_4(q)$. In the
remaining cases, by Lemma \ref{2.2}, for any two maximal parabolic
subgroups $P_1$ and $P_2$ (containing a common Borel subgroup)
there exists a conjugate of $x$ that is contained in
$P_i\backslash U_i$ for either $i=1$ or $2$. The parabolic
subgroups can be chosen so that the Levi complement has only one
component, and since $q \ge 5$, it will always be almost simple.
It follows that since $(xU_i,P_i/U_i)$ will be contained in
$\mathcal A$, $(x,G)$ cannot be a minimal counterexample. Table
\ref{table:unipotentnot3} describes the parabolic subgroups to
choose and the possibilities for the Levi complement.
\begin{table}
    \centering
\begin{tabular}{|c|c|c|}
    \hline
$G_0$ & Nodes corresponding to $P_1$ and $P_2$ & Levi complement type  \\
      & (Bourbaki notation used where node is specified) &\\
\hline
   $A_{l}(q)$, $l \ge 2$  &   end nodes     &$A_{l-1}(q)$\\
    $B_2(q)$ &      end nodes    &$A_1(q)$\\
    $B_l(q)$, $l \ge 3$ &   end nodes  &$B_{l-1}(q)$, $A_{l-1}(q)$   \\
   $C_l(q)$, $l \ge 3$  & end nodes&$C_{l-1}(q)$, $A_{l-1}(q)$  \\
    $D_4(q)$ &  any $2$ end nodes   &   $A_{3}(q)$ \\
$D_l(q)$, $l \ge 5$ &   any $2$ end nodes&  $D_{l-1}(q)$, $A_{l-1}(q)$ \\
 ${^2}A_l(q)$, $l \ge 3$, $l$ odd  &    end and middle node& $A_{(l-1)/2}(q^2)$,
${^2}A_{l-2}(q)$    \\
  ${^2}A_l(q)$, $l \ge 4$, $l$ even   &  end and middle node
          &  ${^2}A_{l-2}(q)$,  $A_{(l-2)/2}(q^2)$             \\
 ${^2}D_4(q)$,    & end nodes &  ${^2}A_{3}(q)$,  $A_{2}(q)$ \\
 ${^2}D_l(q)$, $l \ge 5$   & end nodes &${^2}D_{l-1}(q)$, $A_{l-2}(q)$\\
  $E_6(q)$   &   nodes $1$ and $6$ & $D_5(q)$
     \\
$E_7(q)$   &     nodes $1$ and $2$ & $D_6(q)$, $A_6(q)$
     \\
$E_8(q)$   &    nodes $1$ and $2$ & $D_7(q)$, $A_7(q)$
     \\
$F_4(q)$   &    end nodes & $B_3(q)$, $C_3(q)$
     \\
$G_2(q)$   &     end nodes  &  $A_1(q)$
     \\
${^2}E_6(q)$   & end nodes & ${^2}D_4(q)$, ${^2}A_5(q)$
     \\
\hline
\end{tabular}
\vspace{+1mm} \caption{\label{table:unipotentnot3}Maximal
parabolic subgroups and their Levi complements used in Lemma
\ref{genunipotents}}
\end{table}
\end{proof}
 The next lemma also eliminates the possibility that $q=3$ for
 classical groups.
\begin{lem}\label{q3}
  If $x$ is an element of order $3$ in a classical group $G$
  defined over $\mathbb{F}_3$ then $(x,G)$ cannot be a minimal
  counterexample.
\end{lem}
\begin{proof}
  The aim is to show that if $x$ is not a long root element then,
unless the dimension of the natural module $V$ is very small,
there exists a subgroup $H$ such that $(x,H)$ is in ${\mathcal
A}$. By \cite[pp.34--38]{Wall}, if $x$ is an element of order $3$
in a classical group over $\mathbb{F}_3$ then $x$ will nearly
always fix an orthogonal decomposition unless $n$ is very small.
Suppose that $x$ has order $3$ in $SL(n,3)$ with $n \ge 5$. Then
there exists $x$-invariant subspaces $U$ and $W$ such that $V=U
\oplus W$. Without loss of generality, it suffices to assume that
the dimension of $U$, $k$ say, is at least $3$ and $x$ acts non
trivially on $U$.  Suppose that $x$ does not act as a transvection
on V. If $x$ does not act as a transvection on $U$ then $(x,G)$
cannot be a minimal counterexample. So assume that $x$ acts a
transvection on $U$. Then $x$ must act non-trivially on $W$. So
the dimension of $W$, $n-k$, is at least $2$ and it suffices to
assume that $x$ acts as a transvection on $W$ also, but since $n
\ge 5$ there is a four dimensional subspace $U^{\prime}$ on which
$x$ acts invariantly and is not a transvection and so
$x_{U^{\prime}}$ is contained in a subgroup of type $GL(4,3)$. Now
suppose that $x$ is contained in a symplectic group $Sp(n,3)$ and
that $x$ is not a symplectic transvection. If $n \ge 8$ then $x$
fixes an orthogonal decomposition $U \perp W$. It suffices to
assume that $x$ acts non-trivially on both $U$ and $W$ otherwise
$(x,G)$ is not a minimal counterexample. Moreover, it suffices to
assume that $x$ acts as a symplectic transvection on $U$ and on
$W$, otherwise we $(x_U,Sp(U))$ or $(x_W,Sp(W))$ is contained in
${\mathcal A}$. So assume that for all $u \in U$ and all $w \in W$
\begin{displaymath}
  x_U : u \rightarrow  u+ \lambda \kappa_U(u,a)a, \: a \in U,
  \kappa(a,a)=0;
\end{displaymath}
\begin{displaymath}
  x_W : w \rightarrow w+ \lambda^{\prime} \kappa_W(w,b)b, \: b \in W,
  \kappa(b,b)=0.
\end{displaymath}
Choose $u \in U$ such that $\kappa_U(u,a) \not= 0$ and $w \in W$
such that $\kappa_W(w,b) \not= 0$. Then $x$ acts invariantly on
the non-degenerate subspace $\langle u,w,a,b \rangle$ and is not a
transvection on it, so $(x,G)$ cannot be a minimal counterexample.
Now suppose that $x$ is contained in a unitary group $SU(n,3)$ and
that $x$ is not a unitary transvection. Suppose that $n \ge 5$.
Then there exists an $x$ invariant orthogonal decomposition $U
\perp W$ and as before, it suffices to assume that $x$ acts
non-trivially on both subspaces. Moreover, there exists $H$ such
that $(x,H)$ is contained in ${\mathcal A}$ unless $x$ acts as a
unitary transvection on both $U$ and $W$. So for all $u \in U$ and
all $w \in W$
\begin{displaymath}
  x_U : u \rightarrow  u+ \lambda \kappa_U(u,a)a, \: a \in U,
  \kappa(a,a)=0;
\end{displaymath}
\begin{displaymath}
  x_W : w \rightarrow  w+ \lambda^{\prime} \kappa_W(w,b)b, \: b \in W,
  \kappa(b,b)=0.
\end{displaymath}
Choose $u$ and $w$, as in the symplectic case, so that $\langle
u,w,a,b \rangle$ is a $4$ dimensional, non-degenerate subspace on
which $x$ acts invariantly, but not as a transvection. Then
$(x,G)$ is not a minimal counterexample in this case either.
Finally, suppose that $x$ is contained in an orthogonal group
$\Omega^{\epsilon}(n,3)$ and $x$ is not a long root element.
Suppose that $n \ge 9$. Then there exists an $x$-invariant
orthogonal decomposition $U \perp W$. It suffices to assume that
the action on $U$ and $W$ is not trivial as in the previous cases.
Since $n \ge 9$, it suffices to assume that the dimension of $U$,
$k$ say, is at least $5$. Then if $(x,G)$ is a minimal
counterexample, $x$ must act as a long root element on $U$. Now
either $x$ acts as a long root element on $W$, or $x$ does not act
as a long root element on $W$ and $n-k \le 4$. In the latter case
one can add dimensions from $U$ to $W$ so that $W$ has dimension
at least $5$ and $W$ is still $x$ invariant and non-degenerate. If
this is done then $x_W$ is contained in an orthogonal group $H$
such that $(x_W,H) \in {\mathcal A}$. In the former case, $n-k \ge
4$ ($W$ has Witt defect $0$ if $n-k =4$) and for all $u \in U$ and
all $w \in W$
\begin{displaymath}
\begin{split}
  x_U : u \rightarrow  u+ \lambda \kappa_U(u,a)b - \lambda \kappa_U(u,b)a; \\
  x_W : w \rightarrow w+ \lambda^{\prime}
  \kappa_W(w,c)d-\lambda^{\prime} \kappa_W(w,d)c
\end{split}
\end{displaymath}
where $a,b \in U$; $c,d \in W$; and
$Q(a)=Q(b)=\kappa_U(a,b)=0=Q(c)=Q(d)=\kappa(c,d)$.\\
\indent If $u_1,u_2 \in U$ are such that
$Q(u_i)=0=\kappa_U(u_1,a)=\kappa_U(u_2,b)$, and $\kappa_U(u_1,b)
\not=0$, $\kappa_U(u_2,a) \not =0$, then $x$ acts invariantly on
the non-degenerate $4$ dimensional subspace $\langle
u_1,u_2,a,b\rangle$. Similarly, take $w_1,w_2 \in W$ such that $x$
acts invariantly on the non-degenerate $4$ dimensional subspace
$\langle w_1,w_2,c,d\rangle$. Then $x$ acts invariantly on the
non-degenerate $8$ dimensional subspace $\langle
u_1,u_2,a,b,w_1,w_2,c,d \rangle$---which has Witt defect $0$---and
$x$ does not act as a long root element on it. So it is enough to
check the classical groups of dimension at most $8$ over
$\mathbb{F}_3$ in MAGMA. \par
 If $x$ is a transvection in $SL(n,3)$ then one can reduce to the
case where $n=3$. Similarly if $x$ is a transvection in $SU(n,3)$
or $Sp(n,3)$ then one can reduce to the case where $x \in SU(3,3)$
or $x \in Sp(4,3)$. If $x$ is a long root element in an orthogonal
group then one can reduce to the six dimensional case but
$P\Omega^{+}(6,3) \cong PSL(4,3)$ and $x$ maps to a transvection
under this isomorphism. We can therefore further reduce to
$SL(3,3)$. By \cite{GS}, there exist three conjugates of $x$ that
generate $G_0$ when $G_0=PSL(3,3)$ or $PSU(3,3)$, and four
conjugates of $x$ that generate $G_0=PSp(4,3)$.
\end{proof}

\section{Case U}
It suffices to assume that $n \ge 3$ and $(n,q) \not= (3,2)$. By
Lemmas \ref{Outer}, \ref{genunipotents} and \ref{q3}, if $(x,G)$
is a minimal counterexample, with $G_0 \cong PSU(n,q)$, then $x$
is a semisimple element in $PGU(n,q)$, or $x$ is unipotent in
$PSU(3,q)$, and $q \ge 5$.

\subsection{$x \in PGU(n,q)$, $p \nmid q$ and $p \nmid (n,q+1)$}

By Lemma \ref{lift}, $x \in PGU(n,q)$ lifts to an element in
$GU(n,q)$ of order $p$, with the same sized conjugacy class.
Without loss of generality, it suffices to assume that
$G=PGU(n,q)$ by Lemma \ref{InnDiag}.  Consider the minimal
polynomial of $x$, $m_x(t)$ say. Observe that $m_x(t)$ divides
$t^p-1$ and $t^p-1/(t-1)$ factors over $\mathbb{F}_{q^2}$ as
\begin{displaymath}
  q_1(t) \ldots q_s(t)
\end{displaymath}
where the $q_i(x)$'s are polynomials of degree $k$ (where $k$ is
the smallest positive integer such that $p \mid q^{2k}-1$).  The same
argument as for the case when $G_0=PSL(n,q)$ shows that $x$ will
leave invariant and act non-trivially and irreducibly on a $k$
dimensional subspace $U$ of $V$. Since $x$ acts irreducibly on
$U$, $U$ is either non-degenerate or totally singular (for if $U$
is not non degenerate then there exists $v \in U$ such that
$\kappa_U(v,u)=0$ for all $u \in U$; but $x$ acts irreducibly on
$U$ so $\kappa_U=0$). If $k \ge 2$, in both cases, consider the
induced isometry $x_U$ of $U$. If $U$ is totally singular then
$x_U$ is contained in a group of type $GL(k,q^2)$ and so $(x,G)$
cannot be a minimal counterexample. If $U$ is non-degenerate then
$x_U$ is contained in a subgroup of type $GU(k,q)$. Observe that
if $U$ is non-degenerate then $k$ is odd. For if $k$ was even
then, since $|GU(k,q)|=q^{k(k-1)/2} \prod_{i=1}^k (q^i- (-1)^i)$,
and $p$ would divide $q^k-1=q^{2j}-1$, contradicting the choice of
$k$. Thus if $k \ge 2$ then unless $x$ acts irreducibly or $q=2$,
$(x_U,GU(k,q))$ is contained in ${\mathcal A}$ and $(x,G)$ cannot
be a minimal counterexample.\\
\indent If $q=2$ then there are exceptions in Table
\ref{exceptions}. If $k > 3$ then $x_U$ is contained in $GU(k,2)$,
and the assumption on $k$ implies that $p \not= 3$, so
$(x_U,GU(k,2))$ is contained in ${\mathcal A}$. If $(k,q)=(3,2)$
then $p=7$. Since $U$ is non-degenerate, $x$ also acts invariantly
on $U^{\perp}$. For $n \ge 7$, if this action is non-scalar then
$x_{U^{\perp}}$ is contained in $GU(n-3,2)$ and
$(x_{U^{\perp}},GU(n-3,2))$ is contained in ${\mathcal A}$. If the
action is scalar then take a non-singular vector $w \in
U^{\perp}$, so that $x$ acts invariantly on $U^{\prime}:= U \oplus
\langle w \rangle$. Therefore $x_{U^{\prime}}$ is contained in
$GU(4,2)$ and, since $p \not=3$, it follows that $(x,G)$ is not a
minimal counterexample.\\
\indent If $k=1$ then $x$ acts invariantly on a $1$-dimensional
non-degenerate or singular subspace $U$. Observe that $q\not=3$
since this would imply that $p=2$. First suppose that $q \not=2$
so that $q \ge 4$. If $U$ is non-degenerate then consider the
action of $x$ on $U^{\perp}$. Either $x$ acts non-trivially on
$U^{\perp}$, in which case $x_{U^{\perp}}$ will be contained in
$GU(n-1,q)$, or $x$ has a scalar action on $U^{\perp}$, in which
case there exists a $2$-dimensional non-degenerate subspace
$U^{\prime}$ such that $x_{U^{\prime}}$ is contained in $GU(2,q)$.
Since $q \ge 4$, $(x,G)$ cannot be a minimal counterexample in any
case. Now suppose that $q=2$ and $k=1$ so that $p=3$. If $x$ has
order $3$ in $GU(n,2)$ then a Sylow $3$-subgroup is contained in a
subgroup of type $GU(1,q) \wr S_n$. So it suffices to assume that
$x$ will lie in a subgroup $GU(1,2) \wr S_n$, and if $n \ge 5$, it
suffices to assume that $x$ is contained in $GU(1,2) \perp \ldots \perp
GU(1,2)$ since otherwise $x$ will be non-trivial in a subgroup of
type $S_n$. Thus for $n \ge 5$, if $x$ is not a reflection then
there exists an $n-1$ dimensional, non-degenerate, $x$-invariant
subspace $U^{\prime}$ such that $x_{U^{\prime}} \in GU(n-1,2)$
with $(x_{U^{\prime}},GU(n-1,2))$ is contained in ${\mathcal A}$.
A MAGMA calculation shows that the only exceptions to the theorem
for $G:=PGU(4,2)$ are reflections of order $3$. If $x$ is a
reflection of order $3$ in $GU(n,2)$ then it suffices to treat the
case where $x$ is contained in $GU(4,2)$. A calculation in MAGMA
shows that there exist $g_1,g_2,g_3\in G$ such that $\langle
x,x^{g_1},x^{g_2}, x^{g_3} \rangle$ is not solvable. \\
\indent If $k=1$ and there is not a $1$-dimensional non-degenerate
$x$ invariant subspace then $U$ is totally singular and $x$ is
contained in a parabolic subgroup. Thus by Lemma \ref{2.2}, it
suffices to assume that $x$ acts non-centrally on each component
of the Levi complement of some maximal parabolic subgroup. The
parabolic subgroups of ${^2}A_m(q)$ have Levi complements of type
${^2}A_{m-2}(q)$, $A_k(q^2){^2}A_{m-2k-2}(q)$ and, if $m$ is odd,
${^2}A_{(m-1)/2}(q^2)$. So $(x,G)$ cannot be a minimal
counterexample unless $m=2$, $(m,q)=(3,2)$, or $(m,q)=(4,2)$. If
$m=2$, then $x$ is a reducible semisimple element in $GU(3,q)$, so
$q \ge 4$ and so $x$ leaves invariant a $2$ dimensional,
non-degenerate subspace $U^{\prime}$. In this case,
$x_{U^{\prime}}$ is contained in $GU(2,q)$ and so
$(x_{U^{\prime}},GU(2,q))$ is contained in ${\mathcal A}$. When
$(m,q)=(3,2)$, $G_0=PSU(4,2)$. When $(m,q)=(4,2)$, $G_0 \cong
PSU(5,2)$. These cases can be excluded using MAGMA. \\
\indent  The remaining case is when $k=n$ and $x$ acts irreducibly
in $GU(n,q)$ where $n \ge 3$ is odd (and $(n,q) \not= (3,2)$). Now
one can use \cite{GPPS} to find the possibilities for a maximal
subgroup $M$ containing $\langle x,x^g \rangle$ in $GU(n,q)$. Note that $x$ is a ppd($n$,$q^2$,$n$)-element, and
that $n \ge 3$ is odd. So since $n$ is not a power of $2$ there
are no $2.5$ examples. Lemma \ref{2.6} implies that $M$ must be a
type $2.1$ or $2.4$(b) subgroup. By \cite{KL}, the only possible
such classical maximal subgroups are of type $GU(n,q_0)$ and
$O_n(q)$ ($q$ odd). The only subgroups of this type which contain
an element of order $p\ge 5$ and are not almost simple modulo
scalars are those of type $O_3(3)$ when $n=q=3$. One can treat
$GU(3,3)$ separately in MAGMA. The only other examples are the
field extension examples (type $2.4$(b)). By \cite{KL} and since
$n$ is odd these are subgroups of type $GU(n/r,q^r)$ where $r$ is
an odd prime. Now unless $n=r$ these subgroups are almost simple
modulo scalars and thus $(x,M)$ is contained in ${\mathcal A}$. If
$n=r$ and $x \in M$, where $M$ is a subgroup of type $GU(1,q^n)$,
then observe that $x$ is contained in only one such maximal
subgroup and Theorem \ref{2M} implies that $(x,G)$ cannot be a
minimal counterexample.

\subsection{ $x \in PGU(n,q)$ and $p \mid (q+1,n)$}

Observe that Lemma \ref{InnDiag} still applies so assume that $G =
PGU(n,q)$. This time some conjugacy classes of order $p$ could
only lift to non-trivial scalars in $GU(n,q)$. If $x$ lifts to an
element of order $p$ in $GU(n,q)$ then apply the same argument as
in the previous section. If not then $x^p$ lifts to a non-trivial
scalar in $GU(n,q)$.  So $x$ will have order $p^mj$ say where $p
\nmid j$, but since $\langle x^j\rangle \le \langle x \rangle$ and
$x^j$ will still have order $p$ in $PGU(n,q)$, it suffices to
assume that $j=1$. So assume that the order of $x$ in $GU(n,q)$ is
$p^m$. The minimal polynomial of $x$, $m_x(t)$ divides
$t^p-\zeta_{p^{m-1}}$ where $\zeta_{p^{m-1}}$ is a primitive
$p^{m-1}$th root of unity in $\mathbb{F}_{q^2}$, and
$p^{m-1} \mid (q+1)$. Since there are $p$th roots of unity, either
$t^p-\zeta_{p^{m-1}}$ splits, or it is irreducible over
$\mathbb{F}_{q^2}$. For if $a$ is a root of the equation
$t^p-\zeta_{p^{m-1}}=0$ contained in some field extension, then
this field extension contains all of the roots, $a\omega,
a\omega^2, \ldots, a\omega^{p-1}$. So $t^p-\zeta_{p^{m-1}}$ will
factor into irreducible polynomials of degree equal to the degree
of the smallest field extension containing $a$. However, $p$ is
prime, so the degree of these polynomials is either $1$ or $p$. If
$t^p- \zeta_{p^{m-1}}$ splits then $m_x(t)|
(t-\zeta_{p^{m}})(t-\zeta_{p^{m}}\omega) \ldots
(t-\zeta_{p^{m}}\omega^{p-1})$, where $\zeta_{p^{m}}$ is a
primitive $p^m$th root of unity in $\mathbb{F}_{q^2}$. However
this would imply that $p^m$ divides $q+1$. For $p^{m-1}\mid (q+1)$,
and since $\zeta_{p^m} \in \mathbb{F}_{q^2}$, $p^m \mid (q-1)(q+1)$,
but $p \nmid q-1$ since $p \ge 3$. This would be a contradiction,
since $z= \zeta_{p^m}I_n$ would lie in $Z(GU(n,q))$, so
$(z^{-1}x)^p = \zeta^{-1}_{p^{m-1}} \zeta_{p^{m-1}}I_n = I_n $,
and $x$ would lift to an element of order $p$. So, it suffices to
assume that $m_x(t) = t^p - \zeta_{p^{m-1}}$ is irreducible over
$\mathbb{F}_{q^2}$. It follows that $x$ has rational canonical
form $\rm{diag}[A_1, \ldots, A_{n/p}]$, where
\begin{displaymath}
  A_i = \left(\begin{matrix}
      & I_{p-1} \\
      \zeta_{p^{m-1}} & \\
    \end{matrix} \right).
\end{displaymath}
Thus $x$ acts irreducibly on a subspace $W$ of dimension $p$. Now
$p^m \mid q^{2p}-1$, and in fact $p^m \mid q^p+1$, since if $p \mid q^p-1$ then
$q^p \equiv 1$ (mod $p$) but also $q^p \equiv q$ (mod $p$) by
Fermat's Little Theorem. Therefore $q\equiv 1$ (mod $p$), and
$p \mid q+1$ which contradicts the assumption that $p\ge
3$. So $p^m$ divides $q^p+1$.\\
\indent Assume that $W$ is non-degenerate since if $W$ was totally
singular then $x_W$ would be contained in $GL(p,q^2)$ and
$(x_W,GL(p,q^2))$ would be contained in ${\mathcal A}$. Thus, if
$x$ does not lift to an element of order $p$ then it suffices to
assume that $n=p$ and that $x$ acts
irreducibly. \\
\indent Since $n=p$, and $p \mid q+1$, one can show that a maximal
subgroup $M$ of $GU(p,q)$ of type $GU(1,q)\wr S_p$ always contains
a Sylow $p$-subgroup of $GU(p,q)$. Thus, it suffices to assume
that $x$ is contained in $M$, the normalizer of a maximal split
torus $T$. Moreover, $x$ is non-trivial in $N_G(T)/T \cong S_p$,
since it acts irreducibly. So if $p \ge 5$ then $x$ cannot be a
minimal counterexample. Now suppose that $p=3$. Then $x$ is an
irreducible element in $GU(3,q)$. The character table of $GU(3,q)$
in \cite{Ennola} and the same argument as when $G_0=PSL(2,q)$
implies that there exists an element $z$ in $GU(3,q)$ of order
$q^2-1$ such that $x$ is conjugate to $xz$. So, if $x^g=xz$ then
$\langle x,x^g\rangle= \langle x,z\rangle$ contains $PSU(3,q)$,
since it cannot be contained in any of the maximal subgroups
described in \cite[Theorem 6.5.3]{GLS}.

\subsection{$x \in PSU(3,q)$ and $p \mid q$, $q \ge 5$}

If $x$ is a unipotent element in $G_0=PSU(3,q)$ then the maximal
subgroups of $G_0$ are described in \cite[Theorem 6.5.3]{GLS} and
Lemma \ref{count} can be applied. By Lemma \ref{q3}, there are no
minimal counterexamples when $q=3$. So assume that $q \ge 5$. If
$x$ is a transvection then it stabilizes a non-degenerate $2$
dimensional subspace, and acts non-trivially on it, so $(x,G)$
cannot be a minimal counterexample. Thus $x$ is not a transvection
and $|C_{PSU(3,q)}(x)|=q^2$. Since $(x,G)$ is a minimal
counterexample, the only possibilities for maximal subgroups $X_i$
containing $x$ are of type $GU(1,q) \wr S_3$ (for $p=3$),
$GU(1,q^3)$ (for $p=3$), and parabolic subgroups. Note that
$PSU(3,2)$ and $PGU(3,2)$ do not contain $x$ since they are
$\{2,3\}$-groups that are only relevant when $3 \nmid q$. There is
only one conjugacy class of each of the given subgroups and
$|x^{PSU(3,q)}\cap X_i|$ is at most $6(q+1)^2$, $3(q^2-q+1)$, and
$q^3-1$ in each case respectively. So
\begin{displaymath}
  \begin{split}
    |G|/|C_{G}(x)|^2 = q^3(q^2-1)(q^3+1)/q^4 = & (q^2-1)(q^3+1)/q \ge \\
    (q^3-1) + (q^2-q+1).3 + (q+1)^2.6 & \ge \sum_{i} |x^G \cap X_i|
  \end{split}
\end{displaymath}
for $q \ge 5$, and thus $(x,G)$ cannot be a minimal counterexample
by Lemma \ref{count}.
\end{proof}

\section{Case S} \label{sec:S}

If $G_0 \cong PSp(n,q)$ then the only case left to prove is when
$x$ is a semisimple element contained in ${\rm Inndiag}(PSp(n,q))
\cong PGSp(n,q)$. Since $|PGSp(n,q):PSp(n,q)|=(2,q-1)$, $x$ must
be contained in $PSp(n,q)$, so suppose that $G=G_0$. Furthermore,
by Lemma \ref{lift}, $x$ always lifts to an element in $Sp(n,q)$
of order $p$. \par
 Let $e$ be the smallest positive integer such that $p \mid q^e -1$. Hence the
minimal polynomial of $x$ will be a product of irreducibles of
degree $e$, and possibly $t-1$. Also, $V$ will have an
$e$-dimensional $x$ invariant subspace $U$, on which $x$ acts
irreducibly. $U$ is either totally singular or non-degenerate.
This depends on $e$:
\begin{itemize}
\item {\it $e$ odd and $e \not=1$}
 If $e$ is odd then $U$ is totally singular since there are no
non-degenerate subspaces of $V$ of odd order. So, if $e \ge 3$
then it suffices to assume that $x$ acts non-trivially on $U$, and
$x_U$ is contained in a subgroup $H$ of type $GL(e,q)$. Clearly
$(x,H)$ is contained in ${\mathcal A}$ in this case.

 \item {\it $e=1$}.  If $e$ is $1$ then $U$ is a $1$ dimensional
totally singular subspace, so $x$ is contained in a parabolic
subgroup. By Lemma \ref{2.2}, it suffices to assume that $x$ acts
non-centrally on all the components of the Levi complement of a
maximal parabolic subgroup. This maximal parabolic subgroup can be
of type $C_{m-1}(q)$ ($m \ge 3$); $A_{k}(q)C_{m-k-1}(q)$ ($m \ge
4, 1\le k \le m-3$ ); $A_{m-1}(q)$; or $A_1(q)A_1(q)$ ($m=3$).
Since $p\mid q-1$, $q$ is at least $4$, thus $(x,G)$ cannot be a
minimal counterexample.

\item {\it $e$ even, $e < n$}  If $U$ is totally singular then
$x_U$ is contained in a subgroup $H$ of type $GL(e,q)$, and
$(x,H)$ is in ${\mathcal A}$ unless $(e,q)=(2,2)$ (if
$(e,q)=(2,3)$ then $p=2$). If $(e,q)=(2,2)$ then it suffices to
assume that $n \ge 8$, since $Sp(4,2) \cong S_6$, and the case
$Sp(6,2)$ can be excluded using MAGMA. Since $U$ is totally
singular, $x$ is contained in a parabolic subgroup so we can use
Lemma \ref{2.2} as in the previous case. It follows that $(x,G)$
cannot be a minimal counterexample in this case either. If $U$ is
non-degenerate then $x_U$ is contained in a subgroup $H$ of type
$Sp(e,q)$. $(x,H)$ is contained in ${\mathcal A}$ for $e \ge 4$
and for $e=2$, $q \ge 4$. If $e=2$, and $q \le 3$ then $q=2$. But
it suffices to assume that $n \ge 6$, since $Sp(4,2) \cong S_6$,
so $x_{U^{\perp}}$ is contained in a subgroup $H$ of type
$Sp(n-e,2)$ and $(x,H)$ is contained in $\mathcal A$.
\end{itemize}

If $x$ acts irreducibly then \cite{GPPS} describes the possible
maximal subgroups of $Sp(n,q)$ that could contain $x$.  It
suffices to assume that $n$ is at least $4$, since $SL(2,q) \cong
Sp(2,q)$.  The only $M$'s of concern are those that contain
ppd($n$,$q$,$n$)-elements. By Lemma \ref{2.6}, it suffices to
assume that $M$ is a subgroup of type $2.1$, $2.4$(b) or $2.5$. If
$M$ were a subgroup of type $2.1$ then so long as $M$ is almost
simple modulo scalars, $(x,M)$ is contained in ${\mathcal A}$. By
\cite{KL}, the only possible such maximal subgroups $M$ are type
$2.1$(b) where $M$ contains $Sp(n,q_0)$; and type $2.1$(d) where
$M$ contains $\Omega^{\epsilon}(n,q_0)$ for $q_0$ even. In these
cases, $M$ is almost simple and $(x,M)$ is contained in ${\mathcal
A}$ unless $(n,q)=(4,2)$. However since $Sp(4,2) \cong S_6$, this
case can be excluded. If $M$ is of type $2.5$ then by \cite{KL},
$M$ would be of type $P.O^{-}(2m,2)$ where $q$ is an odd prime,
$n=2^m$, and $P$ is a $2$-subgroup. However since $x$ has odd
order, $xO_2(M)$ would be non-trivial in the quotient $M/O_2(M)$.
Moreover, $e \ge 4$ implies that $m \ge 2$ and thus $M/O_2(M)$ is
almost simple of type $O^{-}(2m,2)$. The only other possibility
for $M$ is to be of type $2.4$(b). In this case, by \cite{KL}, $M$
would be of type $Sp(n/b,q^b)$, where $b$ is a prime and $n/b$ is
even; or of type $GU(n/2,q)$. However, since $n \ge 4$, these are
all almost simple modulo scalars unless $(n,q)=(4,2)$, $(4,3)$, or
$(6,2)$. These exceptions are not a problem since $Sp(4,2) \cong
S_6$, $PSp(4,3) \cong PSU(4,2)$ ($p \not=3$ since $x$ is a
ppd($4$,$3$,$4$)-element) and there are no elements of prime order
in $Sp(6,2)$ that act irreducibly.

\section{Case O} \label{sec:O}

It suffices to
assume that $n \ge 7$ since otherwise $G_0$ is isomorphic to one
of the classical groups that have already been considered.
 If $x \in {\rm Inndiag}(P\Omega_n^{\epsilon}(q))$ has odd prime
order then $x \in P\Omega_n^{\epsilon}(q)$. By Lemma \ref{lift},
$x$ lifts to an element of order $p$ in $\Omega_n^{\epsilon}(q)$.
Lemmas \ref{Outer}, \ref{genunipotents}, and \ref{q3} imply that
if $(x,G)$ is a minimal counterexample then $x \in {\rm
Inndiag}(G_0)$ and $x$ is semisimple. \\
\indent Let $e$ be minimal such that $p \mid q^e -1$, so there exists
an $e$-dimensional subspace $U$ on which $x$ acts invariantly and
irreducibly. Consider the different values for $e$:

\begin{itemize}
\item {\it $e$ odd, $e \ge 3$.} If $e$ is odd then $p \nmid
|O(e,q)|$ so $U$ must be totally singular. It follows that $x_U$
is contained in a subgroup $H$ of type $GL(e,q)$ and $(x_U,H) \in
{\mathcal A}$.

\item {\it $e=1$.}  If $e=1$ then $q \ge 4$ since $p \mid q-1$. If $x$
acts invariantly on a non-degenerate $1$-dimensional subspace $U$
then consider the action of $x$ on $U^{\perp}$. If this action is
non-scalar then $(x,G)$ is not a minimal counterexample since
$x_{U^{\perp}}$ is contained in a subgroup $H$ of type
$O^{\epsilon}(n-1,q)$ and $(x_{U^{\perp}},H)$ is contained in
${\mathcal A}$ since $n\ge 7$. If the action is scalar, then there
exists a $3$-dimensional subspace $Y$ of $U^{\perp}$ such that
$U^{\prime}:=U \oplus Y$ is non-degenerate and $x$ invariant. In
this case, $x_{U^{\prime}}$ will be contained in a subgroup $H$ of
type $O^{\epsilon}(4,q)$. In particular, $(x_{U^{\prime}},H)$
would be contained in ${\mathcal A}$. If $x$ acts invariantly on a
singular, $1$-dimensional subspace then $x$ is contained in a
parabolic subgroup. Thus, by Lemma \ref{2.2}, it suffices to
assume that $x$ acts non-centrally on each component of the Levi
complement of some maximal parabolic subgroup. The possible types
of maximal parabolic subgroup are: $A_{m-1}(q)$, or $B_{m-1}(q)$
if $G_0=B_{m}(q)$; $D_{m-1}(q)$, $A_{m-1}(q)$,
$A_{m-3}(q)A_1(q)A_1(q)$, or $A_k(q)D_{m-k-1}(q)$ if $G_0=D_m(q)$;
or ${^2}D_{m-1}(q)$, $A_{m-2}(q)$, $A_k(q){^2}D_{m-k-1}(q)$, or
$A_{m-3}(q)A_1(q^2)$ if $G_0={^2}D_m(q)$. Since if $G_0=B_m(q)$ then $m \ge 3$ and in the other cases $m \ge 4$, it follows that $(x,G)$ cannot be a minimal
counterexample.

\item {\it $e=2$}.  If $e=2$ then $p \mid q+1$. If $U$ is totally
singular then $x$ is contained in parabolic subgroup and Lemma
\ref{2.2} is applied as above. If $G_0=B_m(q)$, then $m \ge 3$ and $q \ge 5$ since $B_m(2^a) \cong
C_m(2^a)$. The only complication is that if $G_0=D_4(2)$ then all
of the components of a parabolic subgroup of type
$A_1(q)A_1(q)A_(q)$ are solvable. One can verify in MAGMA that
there are no counterexamples when $G_0=D_4(2)$. Now suppose that
$x$ acts invariantly on a $2$-dimensional non-degenerate subspace
$U$. Then $U$ will be anisotropic because of the order of $x$. If
the action of $x$ on $U^{\perp}$ is non-scalar then
$x_{U^{\perp}}$ will be contained in a subgroup $H$ of type
$O^{-\epsilon}(n-2,q)$ (since $U$ has Witt defect $1$,
\cite[4.1.6]{KL}) and $(x_{U^{\perp}},H)$ will be contained in
${\mathcal A}$. Suppose that $x$ acts as a scalar on $U^{\perp}$.
In this case, let $W$ be a 4-dimensional non-degenerate subspace
of $U^{\perp}$ (of Witt defect $0$). Then $x$ will act invariantly
on the non-degenerate space $U^{\prime}=U \oplus W$. So
$x_{U^{\prime}}$ will be contained in a subgroup $H$ of type
$O^-(6,q)$, and $(x_{U^{\prime}},H)$ will be contained in
${\mathcal A}$.

\item {\it $e$ even, $e \ge 4$.} If $e$ is even then $p \mid 
q^{e/2}+1$. Suppose that $U$ is totally singular. Then $x_U$ will
be contained in a subgroup $H$ of type $GL(e,q)$, and $(x_U,H)$
will be contained in ${\mathcal A}$ since $e \ge 4$. So assume
that $U$ is non-degenerate. If $e \not = n$ then $x_U$ will lie in
a subgroup $H$ of type $O^-(e,q)$, with $(x_U,H)$ contained in
${\mathcal A}$.  The only case left to consider is where $x$ acts
irreducibly on $O^-(e,q)$.
\end{itemize}

Since $e=n$ is even it suffices to assume that $n \ge 8$. One can
use \cite{KL} and \cite{GPPS} to find the possible maximal
overgroups of $x$. Lemma \ref{2.6} implies that $M$ must be a
subgroup of type $2.1$, $2.4$(b) or $2.5$. The only subgroups $M$
of type $2.1$ are of type $O^{-}(n,q_0)$, and if $M$ was such a
subgroup then $(x,M)$ would be contained in $\mathcal A$.  There
are no symplectic type normalizer maximal subgroups in
$O^{-}(n,q)$, so there are no $2.5$ type maximal subgroups. This
leaves field extension examples of type $2.4$. The possibilities
are subgroups of type $GU(n/2,q)$, $O^{-}(n/2,q^2)$, and
$O^{-}(n/r,q^r)$ for $r$ a prime and $e/r \ge 4$. All of these are
almost simple modulo scalars and $(x,M)$ would be contained in
${\mathcal A}$. Thus, $(x,G)$ cannot be a minimal counterexample.

\section{$E_l(q)$}

Now suppose that $G_0$ is an exceptional group of type $E_l(q)$,
for $l=6$,$7$, or $8$. If $(x,G)$ is a minimal counterexample then
by Lemmas \ref{Outer} and \ref{genunipotents} either $x \in G_0$
and $p=q=3$, or $x \in {\rm Inndiag}(G_0)$ and
$p \nmid q$. \\
\indent First suppose that $p=q=3$. If $x$ is a long root element
then $\left\langle x, x^g \right\rangle$ is either a $3$-group or
a fundamental $SL(2,3)$ subgroup, by \cite[Proposition
3.2.9]{GLS}. The unipotent conjugacy classes are described in
\cite{Mizuno1,Mizuno2}. Tables \ref{table:E63}, \ref{table:E73},
and \ref{table:E83} list the representatives for the unipotent
classes of order $3$ in $E_l(3)$, and describe a subsystem
subgroup $H$ containing each representative. The tables show that
there are no minimal counterexamples when $x$ is unipotent.
\begin{small}
\begin{landscape}
\begin{table}
    \begin{minipage}{\textwidth}
      \begin{center}
        \begin{tabular}{|l|l|l|}
          \hline
          Representative in $E_6(3)$& Roots generating & Subsystem  \\
          × & subsystem & type  \\
          \hline
          $x_{100000}(1)$ \footnote{In this case, $x$ is a long root element in $A_5(3)$
            and so we
            can find $g_1,g_2$ such that $\langle x, x^{g_1}, x^{g_2}\rangle$ is
not
            solvable} & $\{100000,001000,000100,000010,000001\}$
          &$A_5(q)$  \\
          $x_{100000}(1)x_{001000}(1)$& $\{100000,001000,000100,000010,000001\}$
&
          $A_5(q)$ \\
          $x_{100000}(1)x_{000100}(1)$ &
$\{100000,001000,000100,000010,000001\}$ &
          $A_5(q)$  \\
          $x_{100000}(1)x_{001000}(1)x_{000010}(1)$ &
          $\{100000,001000,000100,000010,000001\}$ &
          $A_5(q)$  \\
          $x_{100000}(1)x_{000100}(1)x_{000001}(1)$ &
          $\{100000,001000,000100,000010,000001\}$ &
          $A_5(q)$  \\
          $x_{100000}(1)x_{001000}(1)x_{000010}(1)x_{000001}(1)$ &
          $\{100000,001000,000100,000010,000001\}$ &
          $A_5(q)$ \\
          $x_{100000}(1)x_{001000}(1)x_{001000}(1)x_{000010}(1)$ &
          $\{100000,001000,000100,000010,010000\}$ &
          $D_5(q)$  \\
          $x_{100000}(1)x_{001000}(1)x_{000010}(1)x_{000001}(1)x_{010000}(1)$ &
          $\{100000,010000,001000,000010,000001\}$ &
          $A_2(q)A_2(q)A_1(q)$  \\
          $x_{100000}(1)x_{000100}(1)x_{000001}(1)x_{122321}(1)$ &
          $\{100000,000100,000010,000001,122321\}$ &
          $A_1(q)A_3(q)A_1(q)$ \\
          \hline
        \end{tabular}
      \end{center}
    \end{minipage}
\vspace{+1mm} \caption{\label{table:E63} Conjugacy classes in
$E_6(3)$ of elements of order $3$}
  \end{table}
\end{landscape}

\begin{table}
  \begin{minipage}{\textwidth}
    \begin{center}
      \begin{tabular}{|l|l|l|}
\hline
        Representative in $E_7(3)$& Roots generating & Subsystem  \\
        × & subsystem & type  \\
\hline
        $x_{34}(1)x_{36}(1)x_{37}(1)x_{38}(1)x_{40}(1)$ &
        $\alpha_{34},\alpha_{40},\alpha_{36},\alpha_{38},\alpha_{37}$ &
        $A_2(q)A_2(q)A_1(q)$ \\
        $x_{34}(1)x_{36}(1)x_{38}(1)x_{40}(1)$ &
        $\alpha_{34},\alpha_{40},\alpha_{36},\alpha_{38}$ &
        $A_2(q)A_2(q)$ \\
        $x_{37}(1)x_{38}(1)x_{39}(1)x_{40}(1)x_{41}(1)$ &
        $\alpha_{37},\alpha_{38},\alpha_{39},\alpha_{40},\alpha_{41}$ &
        $A_1(q)^2A_2(q)A_1(q)$ \\
        $x_{42}(1)x_{43}(1)x_{44}(1)x_{45}(1)$ &
        $\alpha_{42},\alpha_{45},\alpha_{43},\alpha_{44}$ &
        $A_2(q)A_1(q)A_1(q)$ \\
        $x_{44}(1)x_{46}(1)x_{49}(1)$ &
        $\alpha_{44},\alpha_{46},\alpha_{49}$ &
        $A_2(q)A_1(q)$ \\
        $x_{42}(1)x_{43}(1)x_{44}(1)x_{51}(\zeta)x_{49}(1)$ &
        $\alpha_{3},\alpha_{5},\alpha_{7},\alpha_{38},\alpha_{49}$ &
        $D_{4}(q)A_1(q)$ \\
        $x_{44}(1)x_{46}(1)$ &
        $\alpha_{44},\alpha_{46}$ & $A_2(q)$ \\
        $x_{42}(1)x_{43}(1)x_{44}(1)x_{51}(\zeta)$ &
        $\alpha_{3},\alpha_{5},\alpha_{7},\alpha_{38}$ &
        $D_{4}(q)$ \\
        $x_{47}(1)x_{48}(1)x_{49}(1)x_{53}(1)$ &
        $\alpha_{3},\alpha_{5},\alpha_{44},\alpha_{53},\alpha_{49}$ &
        $A_{3}(q)A_1(q)A_1(q)$ \\
        $x_{47}(\zeta)x_{48}(1)x_{49}(1)x_{53}(1)$ &
        $\alpha_{3},\alpha_{5},\alpha_{44},\alpha_{53},\alpha_{49}$ &
        $A_{3}(q)A_1(q)A_1(q)$ \\
        $x_{47}(1)x_{48}(1)x_{49}(1)$ &
        $\alpha_{3},\alpha_{5},\alpha_{44},\alpha_{49}$ &
        $A_{3}(q)A_1(q)$ \\
        $x_{47}(\zeta)x_{48}(1)x_{49}(1)$ &
        $\alpha_{3},\alpha_{5},\alpha_{44},\alpha_{49}$ &
        $A_{3}(q)A_1(q)$ \\
        $x_{53}(1)x_{54}(1)x_{55}(1)$ &
        $\alpha_{2},\alpha_{7},\alpha_{50},\alpha_{55}$ &
        $A_{3}(q)A_1(q)$ \\
        $x_{58}(1)x_{59}(1)$ &
        $\alpha_{2},\alpha_{5},\alpha_{57}$ &
        $A_{3}(q)$ \\
        $x_{63}(1)$ \footnote{In this case, $x$ is a long root element in $A_2(3)$ and so
          we
          can find $g_1,g_2$ such that $\langle x, x^{g_1}, x^{g_2}\rangle$ is
not
          solvable} &
        $\alpha_{1},\alpha_{62}$ &
        $A_{2}(q)$ \\
\hline
      \end{tabular}
    \end{center}
\vspace{+1mm} \caption{\label{table:E73}Conjugacy classes in
$E_7(3)$ of elements of order $3$}
  \end{minipage}
\end{table}

\begin{landscape}
  \begin{table}[htp]
    \centering
    \begin{minipage}{\textwidth}
      \begin{tabular}{|l|l|l|}
        \hline
Representative in $E_8(3)$& Roots generating & Subsystem  \\
        × & subsystem & type  \\
\hline
        $x_{53}(1)x_{54}(1)x_{55}(1)x_{117}(1)x_{118}(1)x_{119}(1)$ &

$\alpha_{53},\alpha_{119},\alpha_{54},\alpha_{55},\alpha_{117},\alpha_{118}$
&
        $A_2(q)^2A_1(q)^2$ \\
        $x_{56}(1)x_{57}(1)x_{117}(1)x_{118}(1)x_{119}(1)$ &
        $\alpha_{56},\alpha_{57},\alpha_{117},\alpha_{118},\alpha_{119}$ &
        $A_2(q)A_2(q)A_1(q)$ \\
        $x_{56}(1)x_{57}(1)x_{117}(1)x_{118}(1)$ &
        $\alpha_{56},\alpha_{57},\alpha_{117},\alpha_{118}$ &
        $A_2(q)A_2(q)$ \\
        $x_{53}(1)x_{54}(1)x_{55}(1)x_{117}(1)x_{124}(\zeta)x_{122}(1)$ &

$\alpha_{53},\alpha_{122},\alpha_{54},\alpha_{55},\alpha_{117},\alpha_{124}$
&
        $A_2(q)A_1(q)^4$ \\
        $x_{58}(1)x_{59}(1)x_{123}(1)x_{124}(1)x_{125}(1)$ &
        $\alpha_{58},\alpha_{59},\alpha_{123},\alpha_{124},\alpha_{125}$ &
        $A_1(q)A_2(q)A_1(q)A_1(q)$ \\
        $x_{60}(1)x_{126}(1)x_{127}(1)x_{128}(1)$ &
        $\alpha_{60},\alpha_{126},\alpha_{127},\alpha_{128}$ &
        $A_2(q)A_1(q)A_1(q)$ \\
        $x_{63}(1)x_{127}(1)x_{130}(1)$ &
        $\alpha_{63},\alpha_{127},\alpha_{130}$ &
        $A_1(q)A_2(q)$ \\
        $x_{63}(1)x_{126}(1)x_{127}(1)x_{128}(1)x_{133}(\zeta)$ &
        $\alpha_{2},\alpha_{5},\alpha_{7},\alpha_{124},\alpha_{63}$ &
        $D_4(q)A_1(q)$ \\
        $x_{63}(1)x_{135}(1)x_{136}(1)x_{137}(1)$ &
        $\alpha_{1},\alpha_{101},\alpha_{62},\alpha_{136},\alpha_{137}$ &
        $A_3(q)A_1(q)A_1(q)$ \\
        $x_{127}(1)x_{130}(1)$ &
        $\alpha_{124},\alpha_{2},\alpha_{5},\alpha_{7}$ &
        $D_4(q)$ \\
        $x_{126}(1)x_{127}(1)x_{128}(1)x_{133}(\zeta)$ &
        $\alpha_{2},\alpha_{5},\alpha_{7},\alpha_{124}$ &
        $D_4(q)$ \\
        $x_{141}(1)x_{142}(1)x_{143}(1)$ &
        $\alpha_{1},\alpha_{6},\alpha_{135},\alpha_{143}$ &
        $A_3(q)A_1(q)$ \\
        $x_{150}(1)x_{151}(1)$ &
        $\alpha_{3},\alpha_{2},\alpha_{148}$ &
        $A_3(q)$ \\
        $x_{157}(1)$ \footnote{In this case, $x$ is a long root element in $A_2(3)$ and
          so there exist $g_1,g_2$ such that $\langle x, x^{g_1}, x^{g_2}\rangle$
          is not solvable} & $\alpha_{8},\alpha_{156}$ & $A_2(q)$ \\
        \hline
      \end{tabular}
    \end{minipage}
\vspace{+1mm} \caption{\label{table:E83} Conjugacy classes in
$E_8(3)$ of elements of order $3$}
  \end{table}
\end{landscape}
\end{small}
The only other possibility is that $p \nmid q$ and  $x \in
\mathrm{Inndiag}(G_0)$. By Lemma \ref{InnDiag}, it suffices to
assume that $G=\textrm{Inndiag}(G_0)$. If $x$ is semisimple then
consider the case where $x$ is contained in a parabolic subgroup.
By Lemma \ref{2.2}, there is a conjugate of $x$ that is contained
in a maximal parabolic that does not centralize any component of
the Levi complement. For $l=6$, $P$ will be of type $D_5(q)$,
$A_1(q)A_4(q)$, or $A_5(q)$, and so $(x,G)$ cannot be a minimal
counterexample. Similarly, for $l=7$ and $8$ one can reduce to a
case where $x$ is acting non-centrally on a component of a Levi
complement. So it suffices to assume that $x$ is not contained in
any parabolic subgroups. In this case, $x$ is semisimple, and
$C_G(x)$ is a reductive group containing no unipotent elements.
Thus, $C_G(x)$ is a torus, and by \cite{Seitz} for example, it
follows that $|C_G(x)| \le (q+1)^l$. The conjugacy classes of
semisimple elements of order $3$ are described in \cite[Table
4.7.3A]{GLS}. So if $x$ is not contained in a parabolic subgroup
then it suffices to assume that $p \ge 5$, since $|C_G(x)| >
(q+1)^l$ for any $x \in E_l(q)$ of order $3$. This observation is
useful since it implies that $x \in O_{\infty}(M)$ for all maximal
subgroups $M$ containing $x$, otherwise $(x,G)$ could not be a
minimal counterexample. If $l=6$ then
\begin{displaymath}
  |G|/|C_G(x)|^2 \ge
  \frac{q^{36}(q^{12}-1)(q^9-1)(q^8-1)(q^6-1)(q^5-1)(q^2-1)}{3(q+1)^{12}},
\end{displaymath}
which is at least $q^{55}$, for $q \ge 2$. The maximal subgroups of
$E_6(q)$ are described in \cite{LSe1} and \cite{LSS}. The possible
maximal subgroups $X_i$ containing $x$ are listed in Table
\ref{table:E6q} together with a crude bound on $|x^G\cap X_i|$.
Clearly the hypotheses of Lemma \ref{count} are satisfied and
there is no minimal counterexample when $l=6$. \par
\begin{table}
\begin{center}
  \begin{tabular}{|l|l|l|}
\hline
    $X_i$   & Bound on $|x^G\cap X_i|$  & Cruder bound \\
\hline
    $d.(L(2,q) \times L(6,q)).de$ & $0$ & $0$ \\
    $e.L(3,q)^3.e^2.S_3$ & $0$ & $q^9$ \\
    $f.(L(3,q^2) \times U(3,q)).g.2$ & $0$ & $0$ \\
    $L(3,q^3).(e \times 3)$ & $0$ & $0$ \\
    $d^2.(P\Omega^{+}(8,q) \times (q-1/d)^2).d^2.S_3$ & $q^2$ & \\
    $({^3}D_4(q) \times (q^2+q+1)).3$ & $q^2+q+1$ & $q^3$ \\
    $h.(P\Omega^{+}(10,q) \times (q-1/h)).h$ & $h(q-1)$ & $q^3$ \\
    $(q-1)^6.W(E_6)$, $q \ge 5$ & $(q-1)^6.51840$ & $q^{13}$ \\
    $(q^2+q+1)^3.3^{1+2}SL(2,3)$ & $(q^2+q+1)^3$ & $q^{9}$ \\
    $3^{3+3}.SL(3,3)$ & $0$ &  \\
\hline
\end{tabular}
\vspace{+1mm} \caption{\label{table:E6q} Bounds on $|x^G\cap X_i|$
for subgroups $X_i$ of $E_6(q)$, $p \ge 5$. $d=(2,q-1)$,
$e=(3,q-1)$, $f=(3,q+1)$, $g=(3,q^2-1)$, $h=(4,q-1)$}
\end{center}
 \end{table}
Now suppose that $l=7$. Then
\begin{displaymath}
  \frac{|G|}{|C_G(x)|^2} \ge
  \frac{q^{63}(q^{18}-1)(q^{14}-1)(q^{12}-1)(q^{10}
    -1)(q^8-1)(q^6-1)(q^2-1)}{2(q+1)^ { 14 }},
\end{displaymath}
which is at least $q^{111}$ for $q \ge 2$.
\begin{table}
  \begin{center}
  \begin{tabular}{|l|l|}
\hline
    $X_i \le E_7(q)$   & Bound on $|x^G\cap X_i|$    \\
    \hline
    $d.(L(2,q) \times P\Omega^{+}(12,q)).d$ & 0  \\
    $f.L^{\epsilon}(8,q).g.(2 \times (2/f)$, $\epsilon=+1$  & 0\\
    $f.L^{\epsilon}(8,q).g.(2 \times (2/f)$, $\epsilon=-1$  & 0\\
    $e.L^{\epsilon}(3,q) \times L^{\epsilon}(6,q).de.2$, $\epsilon=+1$ & 0\\
$e.L^{\epsilon}(3,q) \times L^{\epsilon}(6,q).de.2$, $\epsilon=-1$ & 0\\
    $d^2.(L(2,q)^3 \times P\Omega^{+}(8,q)).d^3.S_3$ & 0 \\
    $(L(2,q^3) \times {^3}D_4(q)).3d$ & 0\\
    $d^3.(L(2,q)^7.d^4.L(3,2))$  & 0\\
    $L(2,q^7).7d$ & 0\\
    $e.(E_6(q) \times (q-1)/e).e.2, \epsilon=1$ & $q$\\
    $e.({^2}E_6(q) \times (q+1)/e).e.2, \epsilon =-1$ & $q^2$\\
    $(q-1)^7.W(E_7)$ & $q^{30}$\\
    $(q+1)^7.W(E_7)$ & $q^{37}$\\
    $(2^2 \times P\Omega^{+}(8,q).2^2).S_3$ & 0\\
    ${^3}D_4(q).3$ & 0\\
\hline
  \end{tabular}
\end{center}
\vspace{+1mm} \caption{\label{table:E7q}Bounds on $|x^G\cap X_i|$
for subgroups $X_i$ of $E_7(q)$, $p \ge 5$. $d=(2,q-1)$,
$e=(3,q-\epsilon)$, $f=(4,q-\epsilon)/d$, $g=(8,q-\epsilon)/d$}
\end{table}
 Table \ref{table:E7q} implies that the hypotheses of Lemma
\ref{count} are satisfied, and there is no minimal counterexample
when $l=7$. If $l=8$ then
\begin{displaymath}
  \begin{split}
   \frac{|G|}{|C_G(x)|^2}  \ge
   \frac{q^{120}(q^{30}-1)(q^{24}-1)(q^{20}-1)(q^{18}-1)(q^{14}-1)(q^{12}
      -1)(q^8-1)(q^2-1)}{2(q+1)^{16}},
  \end{split}
\end{displaymath}
which is at least $q^{239}$ for $q \ge 2$.
\begin{table}
\begin{center}
  \begin{tabular}{|l|l|}
\hline
    $X_i \le E_8(q)$   & Bound on $|x^G\cap X_i|$   \\
    \hline
    $d.P\Omega^{+}(16,q).d$ & $0$ \\
    $d.(L(2,q)\times E_7(q)).d$ & $0$  \\
    $f.(L^{\epsilon}(9,q)).e.2$, $\epsilon=+1$ & $0$ \\
    $f.(L^{\epsilon}(9,q)).e.2$, $\epsilon=-1$ & $0$ \\
    $e.(L^{\epsilon}(3,q)\times E_6^{\epsilon}(q)).e.2$, $\epsilon=+1$ &$0$  \\
    $e.(L^{\epsilon}(3,q)\times E_6^{\epsilon}(q)).e.2$, $\epsilon=-1$ &$0$  \\
    $g.(L^{\epsilon}(5,q))^2.g.4, \epsilon=+1$ & $5$  \\
    $g.(L^{\epsilon}(5,q))^2.g.4, \epsilon=-1$ & $5$  \\
    $SU(5,q^2).4$ & $0$ \\
    $PGU(5,q^2).4$ & $0$ \\
    $d^2.(P\Omega^{+}(8,q))^2.d^2.(S_3 \times 2)$ & $0$ \\
    $d^2.(P\Omega^{+}(8,q^2)).(S_3 \times 2)$ & $0$ \\
    $({^3}D_4(q))^2.6$ & $0$ \\
    $({^3}D_4(q^2)).6$ & $0$  \\
    $e^2.L^{\epsilon}(3,q)^4.e^2.GL(2,3)$, $\epsilon=+1$ & $0$ \\
    $e^2.L^{\epsilon}(3,q)^4.e^2.GL(2,3)$, $\epsilon=-1$ & $0$ \\
    $U(3,q^2)^2.8$ & $0$  \\
    $U(3,q^4).8$ & $0$  \\
    $d^4.L(2,q)^8.d^4.AGL(3,2)$, $q>2$ & $0$  \\
    $(q-1)^8.W(E_8)$ & $q^{46}$  \\
    $(q+1)^8.W(E_8)$ & $q^{46}$  \\
    $(q^4+q^3+q^2+q+1)^2.(5 \times SL(2,5))$ & $5q^{10}$  \\
    $(q^2+q+1)^4.2.(3 \times U(4,2))$ & $q^{15}$ \\
    $(q^2+1)^4.(4 \circ 2^{1+4}).A_6.2$ & $q^{12}$  \\
    $q^8+q^7-q^5-q^4-q^3+q+1.\mathbb{Z}_{30}$ & $q^{15}$  \\
    $(q^4-q^2+1)^2.(\mathbb{Z}_{12}\circ GL(2,3))$ & $q^{10}$ \\
    $(q^8-q^7+q^5-q^4+q^3-q+1).\mathbb{Z}_{30}$ & $q^{15}$  \\
    $(q^4-q^3+q^2-q+1)^2.(5 \times SL(2,5))$ & $5q^{10}$  \\
    $(q^2-q+1)^4.2.(3 \times U(4,2))$ & $q^{12}$ \\
    $2^{5+10}.SL(5,2)$ (exotic) & $0$  \\
    $5^{3}.SL(3,5)$ (exotic) & $q^{7}$ \\
\hline
  \end{tabular}
\end{center}
\vspace{+1mm} \caption{\label{table:E8q}Bounds on $|x^G \cap X_i|$
for $X_i$ a subgroup of $E_8(q)$, $p \ge 5$. $d=(2,q-1)$,
$e=(3,q-\epsilon)$, $f=(9,q-\epsilon)/e$, $g=(5,q-\epsilon)$}
\end{table}
Table \ref{table:E8q} shows that the hypotheses of Lemma
\ref{count} are satisfied, and $(x,G)$ cannot be a minimal
counterexample.

\section{${^2}E_6(q)$}

If $x$ is unipotent then Lemma \ref{genunipotents} implies that
$p=3$. For $q=3$, the unipotent class representatives were
obtained from Frank L\"ubeck, using CHEVIE (\cite{CHEVIE}). From
the class representatives, one can deduce that $(x,G)$ cannot be a
minimal counterexample. If $x$ is semisimple and contained in a
maximal parabolic subgroup then, by Lemma \ref{2.2}, it suffices
to assume that $x$ acts non centrally on all of the components of
the Levi complement. If this parabolic is an end node parabolic
then the Levi complement is of type ${^2}D_4(q)$ or ${^2}A_5(q)$.
If $P$ is not an end-node parabolic then it can be either of type
$A_1(q^2)A_2(q)$ or $A_1(q)A_2(q^2)$. Thus, $(x,G)$ cannot be a
minimal counterexample if $x$ is contained in a parabolic subgroup.\\
\indent So suppose that $x$ is semisimple, and does not lie in any
parabolic subgroups. Then $C_G(x)$ is a torus, and as in the
previous section, note that if $x$ has order $3$ then $|C_G(x)| >
(q+1)^6$ (by \cite[Table 4.7.3A]{GLS}). Thus, by \cite{Seitz}, it
suffices to assume that $p \ge 5$. Moreover,
\begin{displaymath}
  |G|/|C_G(x)|^2 \ge \frac{q^{36}
    (q^{12}-1)(q^9+1)(q^8-1)(q^6-1)(q^5+1)(q^2-1)}{3(q+1)^{12}},
\end{displaymath}
which is at least $q^{55}$ for $q \ge 2$. The possible maximal
subgroups containing $x$ are given in Table \ref{table:2E6q}.
\begin{table}
\begin{center}
  \begin{tabular}{|l|l|l|}
\hline
  $X_i$   & Bound on $|x^G\cap X_i|$  &  Cruder Bound \\
\hline
  $d.(L(2,q) \times U(6,q).de$  & $0$ &  \\
  $e.(U(3,q)^3.e^2.S_3$ & $0$ &$0$  \\
  $f.L(3,q^2) \times L(3,q).g.2$ & $0$ &  \\
  $U(3,q^3).(e \times 3)$ & $0$ &  \\
  $d^2.(P\Omega^{+}(8,q) \times (q+1/d)^2).d^2.S_3$ & $(q+1)^2$ &
  \\
  $({^3}D_4(q) \times (q^2-q+1)).3$ & $(q^2-q+1)$ &  \\
  $h.(P\Omega^{-}(10,q) \times (q+1/h)).h$ & $(q+1)$ & $q^2$ \\
  $(q+1)^6.W(E_6)$, $q \ge 5$ & $(q+1)^6.51840$ & $q^{14}$ \\
  $(q^2-q+1)^3.3^{1+2}SL(2,3)$ & $(q^2-q+1)^3$ & $q^{6}$ \\
  $3^{3+3}.SL(3,3)$ & $0$ &  \\
\hline
\end{tabular}
\end{center}
\vspace{+1mm} \caption{\label{table:2E6q}Bounds on $|x^G\cap X_i|$
for subgroups $X_i$ of ${^2}E_6(q)$, $p \ge 5$.  $d=(2,q-1)$,
$e=(3,q+1)$, $f=(3,q-1)$, $g=(3,q^2-1)$, $h=(4,q+1)$}
\end{table}
Again, the hypothesis of Lemma \ref{count} holds and $(x,G)$
cannot be a minimal counterexample.

\section{$F_4(q)$}

Observe that $\rm{Inndiag}(G_0)=G_0$. If $x$ is unipotent then
$q=3$, and \cite{Lawther} and \cite{Shoji} contain representatives
for the classes of elements of order $3$. They are listed in Table
\ref{table:F43} together with subsystem overgroups of $x$ that
show that $(x,G)$ cannot be a minimal counterexample.\\
\begin{table}
  \begin{minipage}{\textwidth}
    \centering
\begin{tabular}{|l|l|l|}
      \hline
Representative in $F_4(3)$& Roots generating & Subsystem  \\
      × & subsystem & type  \\
\hline
      $x_1=x_{1+2}(1)$ \footnote{In this case, $x$ is a long root element in $A_2(3)$ and so
      there exist $g_1,g_2$ such that $\langle x, x^{g_1}, x^{g_2}\rangle$ is not
      solvable} & $1,1+3$ & $A_2(3)$ \\
      $x_2=x_{1-2}(1)x_{1+2}(-1)$ & $1-2,2-3,3-4,4$ & $B_4(3)$ \\
      $x_3=x_{1-2}(1)x_{1+2}(-\eta)$ & $1-2,2-3,3-4,4$ & $B_4(3)$ \\
      $x_4=x_{2}(1)x_{3+4}(1)$ & $2-3,3-4,4$ & $B_3(3)$ \\
      $x_5=x_{2-3}(1)x_{4}(1)x_{2+3}(1)$ & $2-3,3-4,4$ & $B_3(3)$ \\
      $x_6=x_{2-3}(1)x_{4}(1)x_{2+3}(\eta)$ & $2-3,3-4,4$ & $B_3(3)$ \\
      $x_7=x_{2}(1)x_{1-2+3+4}(1)$ & $2,1-2+3+4$ & $A_2(3)$ \\
      $x_8=x_{2-3}(1)x_{4}(1)x_{1-2}(1)$ & $2-3,1-2,4$ & $A_2(3)A_1(3)$ \\
      $x_{11}=x_{2+3}(1)x_{1+2-3-4}(1)x_{1-2+3+4}(1)$ & $2+3,1+2-3-4,$  &
      $A_1(3)A_2(3)$ \\
      × & $1-2+3+4$ & × \\
\hline
    \end{tabular}
  \end{minipage}
\vspace{+1mm} \caption{\label{table:F43}Conjugacy class
representatives in $F_4(3)$}
\end{table}
 \indent If $x$ is semisimple and contained in a parabolic
 subgroup then Lemma \ref{2.2} implies that $x$ acts non-trivially
 on all of the components of the Levi complement of some parabolic
 $P$. If $P$ is an end node parabolic subgroup the Levi complement
 is of type $B_3(q)$ or $C_3(q)$. If $P$ is not an end node
 parabolic subgroup then $P$ is of type $A_1(q) A_2(q)$. It
 suffices to assume that $x$ does not centralize the $A_1(q)$ or
 $A_2(q)$ components so $(x,G)$ cannot be a minimal
 counterexample. Now suppose that $x$ does not lie in any
 parabolic subgroups. Then $|C_G(x)| \le (q+1)^4$ by \cite{Seitz}.
 However, by \cite[Table 4.7.3A]{GLS}, this condition implies that
 $p \not= 3$. So suppose that $p \ge 5$ and note that
\begin{displaymath}
  |G|/|C_G(x)|^2 \ge q^{24}(q^{12}-1)(q^{8}-1)(q^{6}-1)(q^{2}-1)/(q+1)^{8},
\end{displaymath}
which  is at least $q^{38}$ for $q \ge 2$.
\begin{table}
\begin{center}
  \begin{tabular}{|l|l|}
    \hline
    Type of $X_i$ in $G$    & Bound on $|x^G\cap X_i|$   \\
    with $G_0= F_4(q)$ &  \\
    \hline
    $2.(L(2,q) \times PSp(6,q)).2$, $q$ odd & $0$  \\
    $d.\Omega(9,q)$, $(2,q_0)$ classes & $0$\\
    $d^2.P\Omega^{+}(8,q).S_3$, $(2,q_0)$  classes & $0$  \\
    ${^3}D_4(q).3$ $(2,p)$ classes & $0$  \\
    $e.(L^{\epsilon}(3,q) \times L^{\epsilon}(3,q)).e.2$, $\epsilon=+1$ &
    $e^2$\\
    $e.(L^{\epsilon}(3,q) \times L^{\epsilon}(3,q)).e.2$, $\epsilon=-1$ &
    $e^2$\\
    $(Sp(4,q) \times Sp(4,q)).2$ & $0$  \\
    $Sp(4,q^2).2$ & $0$  \\
    $(q-1)^4.W(F_4)$ , $q=2^a$, $a>2$ & $q^{8}$ \\
    $(q+1)^4.W(F_4)$ , $q=2^a$, $a>1$ & $q^{11}$ \\
    $(q^2+q+1)^2.(3 \times SL(2,3))$, $q=2^a$ & $q^6$  \\
    $(q^2-q+1)^2.(3 \times SL(2,3))$, $q=2^a$, $a>1$ & $q^{6}$  \\
    $(q^2+1)^2.(\mathbb{Z}_{30} \circ GL(2,3))$, $q=2^a$, $a>1$ & $q^{9}$  \\
    $(q^4-q^2+1).\mathbb{Z}_{30}$, $q=2^a$, $a>1$ & $q^{7}$  \\
    $3^3.SL(3,3)$, $q_0 \ge 5$ & $0$  \\
    \hline
  \end{tabular}
\end{center}
\vspace{+1mm} \caption{\label{table:F4q}Bounds on $|x^G\cap X_i|$
for subgroups $X_i$ of $F_4(q)$, $p \ge 5$.
 $d=(2,q-1)$, $e=(3,q-\epsilon)$}
\end{table}
It is clear from Table \ref{table:F4q} that $(x,G)$ cannot be a
minimal counterexample in this case either.

\section{${^2}F_4(2^a)^{\prime}$, where $a$ is odd}

Suppose that $a>1$. Since $p \not= 2$, $x$ is semisimple . If $x$
is contained in a parabolic subgroup then Lemma \ref{2.2} can be
applied. If the resulting subgroup is an end node parabolic
subgroup then it will be of type ${^2}B_2(2^a)$ in which case
$(x,G)$ cannot be a minimal counterexample. If $P$ is not an end
node parabolic then the Levi complement will be of type
$A_1(2^{2a})$ so $(x,G)$ cannot be a minimal counterexample in
this case either. So suppose that $x$ is not contained in any
parabolic subgroups. Then $p \ge 5$, by the same argument as for
$F_4(q)$, and
\begin{displaymath}
  |G|/|C_G(x)|^2 \ge
  q^{12}(q^{6}+1)(q^{4}-1)(q^{3}+1)(q-1)/2(q+1)^{8}.
\end{displaymath}
This is at least $q^{15}$ for $q \ge 8$. The maximal subgroups are
given in \cite{2F4} and include the calculations in Table
\ref{table:2F4q}.
\begin{table}
\begin{center}
  \begin{tabular}{|l|l|}
    \hline
Type of $X_i$ in $G$ with $G_0= {^2}F_4(q)$   & Bound on $|x^G\cap X_i|$ \\
    \hline
    $SU(3,q).2$ & 0 \\
    $PGU(3,q).2$ & 0 \\
    $({^2}B_2(q) \times {^2}B_2(q)).2$ & 0  \\
    $Sp(4,q).2$& 0  \\
    $B_2(q):2$  & 0 \\
    ${^2}F_4(q_0)$  & 0 \\
    $(q+1)^2.GL(2,3)$ & $q^{4}$ \\
    $(q+ \sqrt{2q}+1)^2.(\mathbb{Z}_4 \circ GL(2,3))$ & $q^{4}$  \\
    $(q- \sqrt{2q}+1)^2.(\mathbb{Z}_4 \circ GL(2,3))$, $q>8$ & $q^2$  \\
    $(q^2+ \sqrt{2q^3} +q + \sqrt{2q}+1).\mathbb{Z}_{12}$ & $q^5$ \\
    $(q^2- \sqrt{2q^3} +q - \sqrt{2q}+1).\mathbb{Z}_{12}$ & $q^2$  \\
    \hline
  \end{tabular}
\end{center}
\vspace{+1mm} \caption{\label{table:2F4q}Bounds on $|x^G\cap X_i|$
for subgroups $X_i$ of ${^2}F_4(q)$, $p \ge 5$. $d=(2,q-1)$,
$e=(3,q-\epsilon)$}
\end{table}

Thus $(x,G)$ cannot be a minimal counterexample for $q \ge 8$. If
$a=1$ then $q=2$ and the possibilities for the order of $x$ are
$3$,$5$, and $13$. There are unique classes of cyclic subgroups of
order $3$, $5$, and $13$, by \cite{Atlas}, thus a conjugate of $x$
is contained in a subgroup isomorphic to $PSL(2,25)$. So $(x,G)$
cannot be a minimal counterexample in this case either. \\
\indent The only outer automorphisms are field
automorphisms. If $x$ is a field automorphism then $x$
normalizes an end node parabolic subgroup and acts
non-trivially on the Levi complement. Therefore,
$(x,G)$ cannot be a minimal counterexample.

\section{$G_2(q)$}

Observe that, since $G_2(2)^{\prime} \cong PSU(3,3)$, it suffices
to assume that $q \not= 2$. First consider the case where $q$ is a
power of $2$; so in particular, $x$ is semisimple. The algebraic
group $G_2$ fixes a non-degenerate quadratic form by
\cite[4.1]{SS} and \cite{LSS2} for example. It follows that any
element of $G_0$ is conjugate to an element of either $SL_3(q):2$
or $SU(3,q):2$. Either $x$ is non-central in one of these groups,
in which case $(x,G)$ is not a minimal counterexample, or $x$ is
central in $SL^{\epsilon}(3,q)$, and therefore is contained in a
parabolic subgroup $P$. In the latter case, applying Lemma
\ref{2.2} implies that it suffices to assume that $x$ acts
non-centrally on the Levi complement, which is of type $A_1(q)$.
So $(x,G)$ is not a minimal counterexample in this case
either.\par
 \indent Now suppose that $q$ is odd. If $x$ is
semisimple then, since it has odd order, it must be contained in
$SL^{\epsilon}(3,q)$ for either $\epsilon= +$ or $\epsilon=-$.
Thus if $x$ is not a central element in this subgroup then $(x,G)$
is not a minimal counterexample. If $x$ is central in the
$SL^{\epsilon}(3,q)$ then $x$ is contained in a parabolic subgroup
$P$. So by Lemma \ref{2.2}, it suffices to assume that $x$ acts
non-centrally on the Levi complement, which is of type $A_1(q)$.
Since $p \mid q-\epsilon$ and $q$ is odd, it follows that $q \ge 5$ and
$(x,G)$ cannot be a minimal counterexample in this case either.
Similarly, if $x$ is unipotent and $q\not= 3$ then Lemma \ref{2.2}
implies that $x$ acts non-trivially on a $A_1(q)$ Levi component
of a parabolic subgroup.\par
 \indent Suppose that $q=3=p$ and that $x$ is not a root element.
It is easily verified using MAGMA that there are two conjugacy
classes of elements of order $3$ (long root elements and short
root elements) that belong in Table \ref{exceptions}. Moreover, in
these cases, there exist $g_1, g_2 \in G_2(3)$ such that $\langle
x,x^{g_1},x^{g_2}\rangle$ is not solvable.

\section{${^3}D_4(q)$}

One can use MAGMA for the cases $q=2$ and $q=3$, so assume from
now on that $q \ge 4$. If $x$ is unipotent then, by Lemma
\ref{2.2}, it suffices to assume that $x$ acts non-centrally on a
Levi component of a parabolic subgroup of type $A_1(q)$, or
$A_1(q^3)$. So, since $q \ge 4$, $(x,G)$ cannot be a minimal
counterexample. Similarly, if $x$ is semisimple and is contained
in a parabolic subgroup then Lemma \ref{2.2} applies, as in the
unipotent case. So it suffices to assume that $x$ is not contained
in any parabolic subgroups. It follows that $C_G(x)$ is a torus
and $|C_G(x)| \le (q+1)^4$ by \cite{Seitz}. Thus,
\begin{displaymath}
\frac{|G|}{|C_G(x)|^2} \ge
\frac{q^{12}(q^{8}+q^4+1)(q^{6}-1)(q^{2}-1)}{(q+1)^{8}},
\end{displaymath}
which is at least $q^{18}$ for $q \ge 4$. As usual, observe that $p
\ge 5$ by \cite{GLS}. The possible maximal subgroups containing
$x$ can be deduced from \cite{3D4} and are listed in Table
\ref{table:3D4q} and Lemma \ref{count} shows that $(x,G)$ cannot
be a minimal counterexample.\par
\begin{table}
\begin{center}
  \begin{tabular}{|l|l|}
\hline
    Type of $X_i$ in $G$ with $G_0= {^3}D_4(q)$   & Bound on $|x^G\cap X_i|$ \\
    \hline
    $G_2(q)$ & $0$ \\
    $PGL^{\epsilon}(3,q)$, $q \equiv \epsilon$ (mod $3$) & $0$ \\
    ${^3}D_4(q_1)$, $q_1^{\alpha}=q$, $\alpha \not= 3 \textrm{ prime } $ & $0$
    \\
    $L(2,q^3) \times L(2,q), q_0=2$ & $0$  \\
    $(SL(2,q^3) \circ SL(2,q)).2$, $q_0 \not=2$ & 0  \\
    $((q^2+q+1) \circ SL(3,q)).(3,q^2+q+1).2$ & $q^3$  \\
    $((q^2-q+1) \circ SU(3,q)).(3,q^2-q+1).2$ & $q^2$ \\
    $(q^2+q+1)^2.SL(2,3)$ & $(q+1)^4$ \\
    $(q^2-q+1)^2.SL(2,3)$ & $q^4$  \\
    $(q^4-q^2+1).4$ & $q^4$   \\
\hline
  \end{tabular}
\end{center}
\caption{\label{table:3D4q}Bounds on $|x^G\cap X_i|$ for subgroups
$X_i$ of ${^3}D_4(q)$, $q \ge 4$, and $p \ge 5$.  $d=(2,q-1)$,
$e=(3,q-\epsilon)$, $f=(3,q^2+\epsilon q +1)$}
\end{table}

\section{${^2}B_2(2^a)$, $a \not=1$ odd}

If $a=1$ then ${^2}B_2(2^a)$ is solvable, so it suffices to assume
that $a \not=1$. The maximal subgroups are described in \cite{2B2}
and are listed in Table \ref{table:2B2q}. Note that
$|G|=q^2(q^2+1)(q-1)$ where $q=2^a$.
\begin{table}
\begin{center}
  \begin{tabular}{|l|l|l|}
   \hline
 Subgroup & Bound on $|x^G \cap M|$ & Comments \\
    \hline
    $H$  & $q^2(q-1)$ & Borel subgroup
    \\
    $D_{2(q-1)}$ & $2(q-1)$ & maximal rank\\
    $N(A_1)$  & $4(q+ \sqrt{2q}+1)$& maximal rank\\
    $N(A_2)$  & $4(q- \sqrt{2q}+1)$& maximal rank \\
    ${^2}B_2(2^{a/b}), b \mid a$, & $q^{2/b}(q^{2/b}+1)(q^{1/b}-1)$ & One class
    \cite[Theorem 10]{2B2}\\
\hline
  \end{tabular}
\end{center}
\caption{\label{table:2B2q}Maximal subgroups of ${^2}B_2(2^a)$}
\end{table}
Also, observe that $p \nmid q$ since $p$ is odd and it suffices to
assume that the only subfield subgroup that can contain $x$ is
${^2}B_2(2)$, since otherwise $(x,G)$ would not be a minimal
counterexample. By \cite[Theorem 4]{2B2}, for example, any element
of odd order in ${^2}B_2(q)$ has its centralizer contained in one
of the cyclic groups of order $q-1, q+ \sqrt{2q}+1$ and $q-
\sqrt{2q}+1$. So there are three mutually exclusive possibilities
for $p$: $p \mid q-1$, $p \mid q+ \sqrt{2q}+1$, and $p \mid q- \sqrt{2q}+1$. If
$p \mid q-1$ then
\begin{displaymath}
  |G|/|C_G(x)|^2 \ge
  q^{2}(q^{2}+1)(q-1)/(q-1)^{2}
\end{displaymath}
and
\begin{displaymath}
  \sum_{i} |x^G \cap X_i| \le q^2(q-1) + 2(q-1) + |{^2}B_2(2)|.
\end{displaymath}
An elementary calculation shows that since $q \ge 8$
\begin{displaymath}
  |G|/|C_{G}(x)|^2 > \sum_{i} |x^G \cap X_i|
\end{displaymath}
and Lemma \ref{count} applies. Similarly if $p \mid q \pm \sqrt{2q}+1$
then
\begin{displaymath}
  \sum_{i} |x^G \cap X_i|  \le 4(q^2 \pm\sqrt{2q}+1) + |{^2}B_2(2)|
\end{displaymath}
and the hypotheses of Lemma \ref{count} are satisfied. Thus,
$(x,G)$ cannot be a minimal counterexample.\\
\indent If $x$ is an outer automorphism then it must be a field
automorphism and the same counting argument as for $PSL(2,q)$
applies. Observe that it suffices to assume that there are no
subfield subgroups among the $H_i$'s except
${^2}B_2(q_0)=C_{G_0}(x)$. If there were, then $x$ would be
contained in $\rm{Aut}({^2}B_2(q^{1/r}))$ for some prime $r \not=
p$ and $(x,G)$ would not be a minimal counterexample. So
\begin{displaymath}
  |G_0|/|C_{G_0}(x)|^2 = q^2(q^2+1)(q-1)/q_0^4(q_0^2+1)^2(q_0-1)^2
\end{displaymath}
and
\begin{displaymath}
  \begin{split}
    1 + \sum_{i=1}^m \frac{|H_i|}{|C_{H_i}(x)|^2} \le &
    1+\frac{q^2(q-1)}{q_0^4(q_0-1)^2}+\frac{2(q-1)}{(q_0-1)^2}+ \\
    &
\frac{4(q+\sqrt{2q}+1)}{(q_0+\sqrt{2q_0}+1)^2}
+\frac{4(q-\sqrt{2q}+1)}{(q_0-\sqrt{2q_0}+1)^2}.
  \end{split}
\end{displaymath}
A computation shows that the required inequality holds for all $q
\ge 2^3$ and all $p \ge 3$.

\section{${^2}G_2(3^a)$, $a \not= 1$ odd}

Observe that if $a=1$ then ${^2}G_2^{\prime}(3) \cong L(2,8)$ so
suppose that $a \not= 1$. Also, $|G|=q^3(q^3 + 1)(q - 1)$ and the
maximal subgroups are given in \cite{G2}, which are listed in
Table \ref{table:2G2q}.
\begin{table}
\begin{center}
  \begin{tabular}{|l|l|}
\hline
Subgroup & Comments \\
    \hline
    $P=[q^3].(q-1)$ & Borel subgroup, only one class  \\
    $2 \times L(2,q), q \ge 27$ & maximal rank  \\
    $(2^2 \times D_{(q+1)/2}):3, q \ge 27$ & maximal rank\\
    $\mathbb{Z}_{q + \sqrt{3q}+1}:\mathbb{Z}_{6}$ & maximal rank \\
    $\mathbb{Z}_{q - \sqrt{3q}+1}:\mathbb{Z}_{6}, q \ge 27$ & maximal rank \\
    ${^2}G_2(q_0), q = q_0^{\alpha}, \alpha \textrm{ prime}$ & ×  \\
\hline
  \end{tabular}
\end{center}
\caption{\label{table:2G2q}Maximal subgroups of ${^2}G_2(3^a)$}
\end{table}
If $p \nmid q=3^a$ then are there three mutually exclusive
possibilities: $p \mid (q^2-1)$, $p \mid q - \sqrt{3q}+1$, and $p \mid q +
\sqrt{3q}+1$. First suppose that $p \mid q^2-1$. Then a Sylow
$p$-subgroup is contained inside a maximal subgroup $2 \times
PSL(2,q)$, so some conjugate of $x$ is contained in $PSL(2,q)$.
Thus, $(x,G)$
cannot be a minimal counterexample. \\
\indent If $p \mid q^2-q+1$ then a Sylow $p$-subgroup is contained in
one of the abelian Hall subgroups of order $q \pm \sqrt{3q}+1$, so
it suffices to assume that $x$ is contained in one of these Hall
subgroups and that $|C_G(x)| = q \pm \sqrt{3q}+1$ (see part (4) of
the main theorem in \cite{Ward}). Then an easy count shows that
the hypotheses of Lemma \ref{count} are satisfied. If $p \mid q$ then
\cite{Ward} shows that there are three conjugacy classes of
elements of order $p=3$. One class contains elements in the center
of a Sylow $3$-subgroup and these elements have centralizers of
order $q^3$. The other two conjugacy classes have centralizers of
order $2q^2$. Elements in these classes centralize an involution
$w$, so they are contained in $C_G(w) \cong L(2,q) \times 2$ and
so $(x,G)$ cannot be a minimal counterexample in this case. Now
\cite{Lawther} gives a representative $x_{2a+b}(1)x_{3a+2b}(1)$
for the conjugacy class of elements $t$ with $|C_G(t)|=q^3$. This
is contained in ${^2}G_2(3) \cong L(2,8):3 $, so $(x,G)$ cannot be
a minimal counterexample in this case either. If $x$ is an outer
automorphism then it must be a field automorphism.  The same
method as for ${^2}B_2(2^a)$ applies here. As before, it suffices
to assume that there are no subfield subgroups among the $H_i$'s,
other than ${^2}B_2(2^{a/p})$. So
\begin{displaymath}
  |G_0|/|C_{G_0}(x)|^2 = q^3(q^3+1)(q-1)/q_0^6(q_0^3+1)^2(q_0-1)^2
\end{displaymath}
and
\begin{displaymath}
  \begin{split}
    1 + \sum_{i=1}^m \frac{|H_i|}{|C_{H_i}(x)|^2} \le &
    1+\frac{q^3(q-1)}{q_0^6(q_0-1)^2}+\frac{6(q+1)}{(q_0+1)^2}+
 \frac{6(q+\sqrt{3q}+1)}{(q_0+\sqrt{3q_0}+1)^2} +
    \frac{6(q-\sqrt{3q}+1)}{(q_0-\sqrt{3q_0}+1)^2} + \\
  &  \frac{2q(q^2-1)}{q_0^2(q_0^2-1)^2}.
  \end{split}
\end{displaymath}
A computation now shows that $(x,G)$ cannot be a minimal
counterexample for any prime power $q$.

\section{Sporadic Groups}

If $G_0$ is one of the following sporadic groups then a MAGMA
calculation shows that there exists $g \in G$ such that $\langle
x, x^g \rangle$ is not solvable:
\begin{displaymath}
 \begin{split}
    M_{11} , M_{12}, M_{22}, M_{23}, & M_{24}, J_1, J_2, J_3, Co_2, \\
    Co_3,McL, HS, Suz, &He, Fi_{22}, Fi_{23}, Fi_{24}.
 \end{split}
\end{displaymath}
There are $9$ remaining sporadic groups, which are a little more
awkward. One can use \cite{Atlas}, which describes the conjugacy
classes and maximal subgroups. In certain circumstance, one can
show that some element of a conjugacy class is contained inside
some smaller almost simple group. In particular, one can do this
if there is a unique conjugacy class of elements of order $p$, or
a multiple of $p$ that powers up to the conjugacy class in
question. Then any almost simple subgroup containing elements of
this order will contain an element of $x^G$, and thus $(x,G)$
cannot be a minimal counterexample. Clearly, this also applies if
all of the conjugacy classes of elements of order $p$ are powers
of each other. In the remaining cases, one can use MAGMA with a
little more care. The details are listed in Tables \ref{table:J4},
\ref{table:Co1}, \ref{table:Ru}, \ref{table:ON}, \ref{table:HN},
\ref{table:Ly}, \ref{table:Th}, \ref{table:B}, and \ref{table:M}.

\begin{table}
\begin{minipage}{\textwidth}
\centering
\begin{tabular}{|c|c|c|c|}
  \hline
  \textrm{Class(es)}  & MAGMA  & $x$ contained in ""  \\
  &&due to power up \\
  \hline
  $3$   &   &  $M_{22}:2, 3$\\
  5   &   &  $M_{22}:2$, 5\\
  7   &   &  $M_{22}:2$, 7 \footnote{$7A=(7B)^3$, $7B=7A^3$}\\
  11A   &    & $PSU(3,11):2, 44$\\
  11B  & &  $\langle x,a \rangle$ generates\footnote{In this case, $a$ is a standard generator in class 2A; $x$ is a standard representative for class 11B; $x^3a$ has order 43 and
$x^2a$ has order 35, so $\langle x,a \rangle$ cannot be contained in any maximal subgroups}  \\
  23  & &    $2^{11}:M_{24}, 23$\\
  29  & &  $\langle x,a \rangle$ generates\footnote{In this case, $x$ is a standard representative for class 29A. We can show in MAGMA that the group order is a multiple of
29.44}\\
  31 & &   $L(2,32)$, $31$ \\
  37  & &  $U(3,11)$, $37$\\
  43 & & $\langle x,a \rangle$ generates\footnote{In this case, $x$ is a standard representative for class 43A; but a calculation in MAGMA shows that 43.23 divides the order of $\langle x,a \rangle$}\\
\hline
\end{tabular}
\vspace{+2mm} \caption{\label{table:J4}Janko group, $J_4$}
\end{minipage}
\end{table}

\begin{table}
 \centering
 \begin{tabular}{|c|c|cc|}
  \hline
  \textrm{Class(es)}  & MAGMA &  & $x$ contained in ""  \\
  &&&due to power up \\
  \hline
  $3$   & done & & \\
  5   & done &  &\\
  7A   & &   &$A_9, 42$\\
7B   & done &   & \\
  11  & &     & $Co_3 ,11$\\
  13  &    & &$3:Suz:2, 13$\\
  23  & &  & $Co_2, 23$ \\
  \hline
\end{tabular}
\vspace{+2mm}
 \caption{\label{table:Co1}Conway group, $Co_1$}
\end{table}

\begin{table}
\centering
 \begin{tabular}{|c|c|cc|}
  \hline
  \textrm{Class(es)}  & MAGMA &  & $x$ contained in ""  \\
  &&&due to power up \\
  \hline
  $3$   & done & & \\
  5   & done &  &\\
  7   & &   & $A_8,7$\\
  13  & &     &$PSL(2,13),13$ \\
  29  & &  &$PSL(2,29), 29$ \\
  \hline
\end{tabular}
\vspace{+2mm} \caption{\label{table:Ru}Rudvalis group, $Ru$}
 \end{table}

\begin{table}
\centering
\begin{tabular}{|c|c|cc|}
  \hline
  \textrm{Class(es)}  & MAGMA &   & $x$ contained in ""  \\
  &&&due to power up \\
  \hline
  $3$   & done & & \\
  5   &  &  &$A_7, 5$\\
  7A   & &  &$PSL(3,7),14$\\
7B  &done &   &\\
  11  & &     & $J_1, 11$\\
  19  & &     & $J_1, 19$ \\
  31  & &   & $PSL(2,31)$\\
  \hline
\end{tabular}
\vspace{+2mm} \caption{\label{table:ON}O'Nan group, $O'N$}
 \end{table}

\begin{table}
\begin{minipage}{\textwidth}
\centering
 \begin{tabular}{|c|c|cc|}
  \hline
  \textrm{Class(es)}  & MAGMA &   & $x$ contained in ""  \\
  &&&due to power up \\
  \hline
  3A   &  & & $A_{12}$,21A \\
  3B   &  &  &  $A_{12}$, 9\\
  5A   & &   & $A_{12}$, 35\\
5B-E &  done \footnote{MAGMA calculation performed using permutation representation in \cite{wwwatlas}} &   &
\\
  7   & &   &$A_{12},7$\\
  11  & &     & $A_{12},11$\\
  19  & &   & $PSU(3,8),19$\\
  \hline
\end{tabular}
\vspace{+2mm}
 \caption{\label{table:HN}Harada--Norton group, $HN$}
\end{minipage}
\end{table}

\begin{table}
\begin{minipage}{\textwidth}
\centering
\begin{tabular}{|c|c|cc|}
  \hline
  \textrm{Class(es)}  & MAGMA &   & $x$ contained in ""  \\
  &&&due to power up \\
  \hline
  3A   &  & & $2.A_{11}$,21A \\
  3B   &  &  &  $2.A_{11}$, 9\\
  5A   & &   & $2.A_{11}$ , 20 \\
  5B   &done \footnote{If $a$ and $x$  are standard representatives for classes 2A and 5B respectively then $ax^b$ has order 67 so $\left\langle a,x^b\right\rangle$
can not be contained in any maximal subgroup} &   & \\
  7   & &   &$2.A_{11},7$\\
  11  & &    &$2.A_{11},11$ \\
  31  & &    & $5^3.PSL(3,5),31$\\
  37  & done \footnote{If $a$ and $x$  are standard representatives for classes 2A and 37A respectively then ax has order 67 so $\left\langle a,x\right\rangle$ can
not be contained in any maximal subgroup}&   & \\
67  & done \footnote{If $a$ and $x$  are standard representatives for classes 2A and 67A respectively then ax has order 14 so $\left\langle
a,x\right\rangle$ can
not be contained in any maximal subgroup}& & \\
  \hline
\end{tabular}
\vspace{+2mm} \caption{\label{table:Ly}Lyons group, $Ly$}
\end{minipage}
 \end{table}

\begin{table}
\centering
\begin{tabular}{|c|cc|c|}
  \hline
  \textrm{Class(es)}  & MAGMA &  & $x$ contained in ""  \\
  &&&due to power up \\
  \hline
  3A   &  & &  $PSU(3,8)$, 21 \\
  3B   &  & &  $A_9$, 9\\
  3C   &  & &  $A_9$, 15\\
  5   & &   &  $A_9$, 5\\
  7   & &   & $A_9$, 7 \\
  13   & &   &$PSL(3,3)$, 13\\
  19  & &     & $PSL(2,19)$,19\\
  31  & &     &$2^5.PSL(5,2), 31$ \\
  \hline
\end{tabular}
\vspace{+2mm} \caption{\label{table:Th}Thompson Group, $Th$}
  \end{table}

\begin{table}
\begin{minipage}{\textwidth}
\centering
\begin{tabular}{|c|cc|c|}
  \hline
  \textrm{Class(es)}  & MAGMA &  & $x$ contained in ""  \\
  &&&due to power up \\
  \hline
  3A   &  & &$HN, 21$  \\
  3B   & &   & $HN,9$ \\
  5A   & &   & $HN, 35$\\
  5B   & &    & $HN$, \footnote{The order of $C_B(x)$ is a multiple of $5^6$, but $5^5 \nmid C_B(y)$ if $y \in 5A$, so any member of the class 5B in $HN$
(centralizer
order $500,000$) must be in the Baby Monster class 5B}\\
  7  & &     & $PSL(2,49)$, 7 \\
  11  & &     & $PSL(2,11)$, 11 \\
  13  & &   & $PSL(3,3)$, 13\\
  17  & &   &$PSL(2,17)$, 17\\
  19 & &   & $HN, 19$\\
  23 & &  &$FI_{23}, 23$\\
  31& &  & $PSL(2,31)$, 31\\
  47 & done \footnote{Since $ax$ has order 9, where $a$ and $x$ are standard representatives for classes 2A and 47A respectively, $\left\langle a,x \right\rangle$ generates since the only maximal
subgroup with order a multiple of $47$ is $47:23$}& &\\
  \hline
\end{tabular}
\vspace{+2mm} \caption{\label{table:B}Baby Monster, $B$}
\end{minipage}
\end{table}

\begin{table}
\begin{minipage}{\textwidth}
\centering
\begin{tabular}{|c|c|c|}
  \hline
  \textrm{Class(es)}  &   $|C_G(x)|$  & $x$ contained in ""  \\
  && due to power up \\
  \hline
  3A   &   & $B, 48$  \\
  3B   &   & $A_{12}, 9$ \\
  3C   &    & $PSU(3,8) \times A_5, 57$\\
  5A   &    & $PSL(2,11) \times M_{12},55$\\
  5B  &     & $PSL(2,25),25$\\
  7A  &    & $(A_5 \times A_{12}),105$\\
7B  &      & contained in $PSL(2,71)$ group by
\cite[pg 596]{anatomy}\\
  11  &       & $2.B, 11$ \\
13A  & 73008&     $S_3 \times Th$ \footnote{Since an element of
order 13 in $Th$ has centralizer order 39, it follows that any
such element is in 13A since $39.6$ divides $|C_G(x)|$ but does
not divide $52728$.}
\\
13B  & 52728    & Lies in $6.Suz$ by \cite[pg 593]{anatomy}\\
17  & &       $2.B, 17$\\
19 & &       $2.B, 19$\\
23  & &       $2.B, 23$\\
29  & &       $3.Fi_{24}, 29$ \\
31  & &       $2.B, 31$ \\
41 &  & $3^8.O^{-}(8,3), 41$ \\
 47 && $2.B, 47$ \\
 59 && $PSL(2,59), 59$ \\
 71 && $PSL(2,71), 71$ \\
  \hline
\end{tabular}\vspace{+2mm}
\caption{\label{table:M}Monster Group, $M$}
\end{minipage}
\end{table}
This completes the proof of Theorem A*.

\bibliographystyle{amsalpha}
\bibliography{bibliography}

\newcommand{\etalchar}[1]{$^{#1}$}
\def\cprime{$'$}
\providecommand{\bysame}{\leavevmode\hbox to3em{\hrulefill}\thinspace}
\providecommand{\MR}{\relax\ifhmode\unskip\space\fi MR }
\providecommand{\MRhref}[2]{%
  \href{http://www.ams.org/mathscinet-getitem?mr=#1}{#2}
}
\providecommand{\href}[2]{#2}
\begin{thebibliography}{GGKP08b}

\bibitem[ABN{\etalchar{+}}]{wwwatlas}
Rachel Abbott, John Bray, Simon Nickerson, Steve Linton, Simon Norton, Richard
  Parker, Ibrahim Suleiman, Jonathan Tripp, Peter Walsh, and Robert Wilson,
  \emph{A www-atlas of finite group representations.}

\bibitem[AG84]{GurAs}
M.~Aschbacher and R.~Guralnick, \emph{Some applications of the first cohomology
  group}, J. Algebra \textbf{90} (1984), no.~2, 446--460. \MR{MR760022
  (86m:20060)}

\bibitem[Bur04]{Bs}
Timothy~C. Burness, \emph{Fixed point spaces in actions of classical algebraic
  groups}, J. Group Theory \textbf{7} (2004), no.~3, 311--346. \MR{MR2063000
  (2005c:14054)}

\bibitem[CCN{\etalchar{+}}85]{Atlas}
J.~H. Conway, R.~T. Curtis, S.~P. Norton, R.~A. Parker, and R.~A. Wilson,
  \emph{Atlas of finite groups}, Oxford University Press, Eynsham, 1985,
  Maximal subgroups and ordinary characters for simple groups, With
  computational assistance from J. G. Thackray. \MR{MR827219 (88g:20025)}

\bibitem[Enn62]{Ennola}
Veikko Ennola, \emph{On the conjugacy classes of the finite unitary groups},
  Ann. Acad. Sci. Fenn. Ser. A I No. \textbf{313} (1962), 13. \MR{MR0139651 (25
  \#3082)}

\bibitem[FGG]{GGFar}
Paul Flavell, Robert Guralnick, and Simon Guest, \emph{Characterizations of the
  solvable radical}, {P}reprint, available at http://ar{X}iv.org.

\bibitem[GGKP08a]{GGKP2}
Nikolai Gordeev, Fritz Grunewald, Boris Kunyavskii, and Eugene Plotkin, \emph{A
  commutator description of the solvable radical of a finite group}, Groups
  Geom. Dyn. \textbf{2} (2008), no.~1, 85--120. \MR{MR2367209 (2008j:20057)}

\bibitem[GGKP08b]{GGKP}
\bysame, \emph{A description of {B}aer--{S}uzuki type of the solvable radical
  of a finite group}, J. Pure and Applied Algebra, to appear (2008).

\bibitem[GHL{\etalchar{+}}96]{CHEVIE}
Meinolf Geck, Gerhard Hiss, Frank L{\"u}beck, Gunter Malle, and G{\"o}tz
  Pfeiffer, \emph{C{HEVIE}---a system for computing and processing generic
  character tables}, Appl. Algebra Engrg. Comm. Comput. \textbf{7} (1996),
  no.~3, 175--210, Computational methods in Lie theory (Essen, 1994).
  \MR{MR1486215 (99m:20017)}

\bibitem[GK00]{GK}
Robert~M. Guralnick and William~M. Kantor, \emph{Probabilistic generation of
  finite simple groups}, J. Algebra \textbf{234} (2000), no.~2, 743--792,
  Special issue in honor of Helmut Wielandt. \MR{MR1800754 (2002f:20038)}

\bibitem[GL83]{GL}
Daniel Gorenstein and Richard Lyons, \emph{The local structure of finite groups
  of characteristic {$2$} type}, Mem. Amer. Math. Soc. \textbf{42} (1983),
  no.~276, vii+731. \MR{MR690900 (84g:20025)}

\bibitem[GLS98]{GLS}
Daniel Gorenstein, Richard Lyons, and Ronald Solomon, \emph{The classification
  of the finite simple groups. {N}umber 3.}, Mathematical Surveys and
  Monographs, vol.~40, American Mathematical Society, Providence, RI, 1998.
  \MR{MR1490581 (98j:20011)}

\bibitem[GPPS99]{GPPS}
Robert Guralnick, Tim Penttila, Cheryl~E. Praeger, and Jan Saxl, \emph{Linear
  groups with orders having certain large prime divisors}, Proc. London Math.
  Soc. (3) \textbf{78} (1999), no.~1, 167--214. \MR{MR1658168 (99m:20113)}

\bibitem[GPS07]{GPS}
Robert Guralnick, Eugene Plotkin, and Aner Shalev, \emph{Burnside-type problems
  related to solvability}, Internat. J. Algebra Comput. \textbf{17} (2007),
  no.~5-6, 1033--1048. \MR{MR2355682 (2008j:20110)}

\bibitem[GS03]{GS}
Robert~M. Guralnick and Jan Saxl, \emph{Generation of finite almost simple
  groups by conjugates}, J. Algebra \textbf{268} (2003), no.~2, 519--571.
  \MR{MR2009321 (2005f:20057)}

\bibitem[Gur98]{Gu}
Robert~M. Guralnick, \emph{Some applications of subgroup structure to
  probabilistic generation and covers of curves}, Algebraic groups and their
  representations (Cambridge, 1997), NATO Adv. Sci. Inst. Ser. C Math. Phys.
  Sci., vol. 517, Kluwer Acad. Publ., Dordrecht, 1998, pp.~301--320.
  \MR{MR1670777 (2000d:20062)}

\bibitem[KL90]{KL}
Peter Kleidman and Martin Liebeck, \emph{The subgroup structure of the finite
  classical groups}, London Mathematical Society Lecture Note Series, vol. 129,
  Cambridge University Press, Cambridge, 1990. \MR{MR1057341 (91g:20001)}

\bibitem[Kle88a]{G2}
Peter~B. Kleidman, \emph{The maximal subgroups of the {C}hevalley groups {$G\sb
  2(q)$} with {$q$} odd, the {R}ee groups {$\sp 2G\sb 2(q)$}, and their
  automorphism groups}, J. Algebra \textbf{117} (1988), no.~1, 30--71.
  \MR{MR955589 (89j:20055)}

\bibitem[Kle88b]{3D4}
\bysame, \emph{The maximal subgroups of the {S}teinberg triality groups {$\sp
  3D\sb 4(q)$} and of their automorphism groups}, J. Algebra \textbf{115}
  (1988), no.~1, 182--199. \MR{MR937609 (89f:20024)}

\bibitem[Law95]{Lawther}
R.~Lawther, \emph{Jordan block sizes of unipotent elements in exceptional
  algebraic groups}, Comm. Algebra \textbf{23} (1995), no.~11, 4125--4156.
  \MR{MR1351124 (96h:20084)}

\bibitem[LS03]{LSe1}
Martin~W. Liebeck and Gary~M. Seitz, \emph{A survey of maximal subgroups of
  exceptional groups of {L}ie type}, Groups, combinatorics \& geometry (Durham,
  2001), World Sci. Publ., River Edge, NJ, 2003, pp.~139--146. \MR{MR1994964
  (2004f:20089)}

\bibitem[LSS92]{LSS}
Martin~W. Liebeck, Jan Saxl, and Gary~M. Seitz, \emph{Subgroups of maximal rank
  in finite exceptional groups of {L}ie type}, Proc. London Math. Soc. (3)
  \textbf{65} (1992), no.~2, 297--325. \MR{MR1168190 (93e:20026)}

\bibitem[LSS96]{LSS2}
\bysame, \emph{Factorizations of simple algebraic groups}, Trans. Amer. Math.
  Soc. \textbf{348} (1996), no.~2, 799--822. \MR{MR1316858 (96g:20064)}

\bibitem[Mal91]{2F4}
Gunter Malle, \emph{The maximal subgroups of {${}\sp 2F\sb 4(q\sp 2)$}}, J.
  Algebra \textbf{139} (1991), no.~1, 52--69. \MR{MR1106340 (92d:20068)}

\bibitem[Miz77]{Mizuno1}
Kenzo Mizuno, \emph{The conjugate classes of {C}hevalley groups of type
  {$E\sb{6}$}}, J. Fac. Sci. Univ. Tokyo Sect. IA Math. \textbf{24} (1977),
  no.~3, 525--563. \MR{MR0486170 (58 \#5951)}

\bibitem[Miz80]{Mizuno2}
\bysame, \emph{The conjugate classes of unipotent elements of the {C}hevalley
  groups {$E\sb{7}$} and {$E\sb{8}$}}, Tokyo J. Math. \textbf{3} (1980), no.~2,
  391--461. \MR{MR605099 (82m:20046)}

\bibitem[NW02]{anatomy}
Simon~P. Norton and Robert~A. Wilson, \emph{Anatomy of the {M}onster. {II}},
  Proc. London Math. Soc. (3) \textbf{84} (2002), no.~3, 581--598.
  \MR{MR1888424 (2003b:20023)}

\bibitem[Sei83]{Seitz}
Gary~M. Seitz, \emph{The root subgroups for maximal tori in finite groups of
  {L}ie type}, Pacific J. Math. \textbf{106} (1983), no.~1, 153--244.
  \MR{MR694680 (84g:20085)}

\bibitem[Sho74]{Shoji}
Toshiaki Shoji, \emph{The conjugacy classes of {C}hevalley groups of type
  {$(F\sb{4})$} over finite fields of characteristic {$p\not=2$}}, J. Fac. Sci.
  Univ. Tokyo Sect. IA Math. \textbf{21} (1974), 1--17. \MR{MR0357641 (50
  \#10109)}

\bibitem[SS97]{SS}
Jan Saxl and Gary~M. Seitz, \emph{Subgroups of algebraic groups containing
  regular unipotent elements}, J. London Math. Soc. (2) \textbf{55} (1997),
  no.~2, 370--386. \MR{MR1438641 (98m:20057)}

\bibitem[Suz62]{2B2}
Michio Suzuki, \emph{On a class of doubly transitive groups}, Ann. of Math. (2)
  \textbf{75} (1962), 105--145. \MR{MR0136646 (25 \#112)}

\bibitem[Tho68]{Thompson}
John~G. Thompson, \emph{Nonsolvable finite groups all of whose local subgroups
  are solvable}, Bull. Amer. Math. Soc. \textbf{74} (1968), 383--437.
  \MR{MR0230809 (37 \#6367)}

\bibitem[Wal63]{Wall}
G.~E. Wall, \emph{On the conjugacy classes of classical groups}, J. Austral.
  Math. Soc. \textbf{3} (1963), 1--62. \MR{MR0150210 (27 \#212)}

\bibitem[War66]{Ward}
Harold~N. Ward, \emph{On {R}ee's series of simple groups}, Trans. Amer. Math.
  Soc. \textbf{121} (1966), 62--89. \MR{MR0197587 (33 \#5752)}

\bibitem[Wie64]{Wiel}
Helmut Wielandt, \emph{Finite permutation groups}, Translated from the German
  by R. Bercov, Academic Press, New York, 1964. \MR{MR0183775 (32 \#1252)}

\end{thebibliography}

\end{document}